\documentclass[12pt,reqno]{amsart}
\usepackage{amsmath,amssymb,amsthm,bbm,fullpage,cite,verbatim}
\usepackage{amssymb}
\usepackage[utf8]{inputenc}
\usepackage{bbm}

\DeclareMathOperator*{\E}{\mathbb{E}}
\DeclareMathOperator*{\Prob}{\mathbb{P}}

\newtheorem{theorem}{Theorem}
\newtheorem{lemma}[theorem]{Lemma}

\newtheorem{corollary}[theorem]{Corollary}

\theoremstyle{definition}

\renewcommand\leq{\leqslant}
\renewcommand\geq{\geqslant}
\renewcommand\le{\leqslant}
\renewcommand\ge{\geqslant}

\begin{document}

\title{Large Sumsets from Small Subsets}
\author{B\'ela Bollob\'as \and Imre Leader \and Marius Tiba}

\address{Department of Pure Mathematics and Mathematical Statistics,
Wilberforce Road,
Cambridge, CB3 0WA, UK, and Department of Mathematical Sciences,
University of Memphis, Memphis, TN 38152, USA}\email{b.bollobas@dpmms.cam.ac.uk}

\address{Department of Pure Mathematics and Mathematical Statistics,
Wilberforce Road, Cambridge, CB3 0WA, UK}\email{i.leader@dpmms.cam.ac.uk}

\address{Department of Pure Mathematics and Mathematical Statistics,
Wilberforce Road, Cambridge, CB3 0WA, UK}\email{mt576@cam.ac.uk}

\thanks{The first author was partially supported by NSF grant DMS-1855745}

\begin{abstract}
In this paper we start to investigate a new body of questions in additive combinatorics.
The fundamental Cauchy--Davenport theorem gives a lower bound on the size of a sumset $A+B$ for subsets
of the cyclic group ${\mathbb Z}_p$ of order $p$ ($p$ prime), and this is just one example of a large
family of results. Our aim in this paper is to investigate what happens if we restrict the number of
elements of one set that we may use to form the sums. Here is the question we set out to answer:
given two subsets, $A$ and $B$, does $B$ have a subset $B'$ of bounded
size such that $A+B'$ is large, perhaps even comparable to the size of $A+B$? In particular,
can we get close to the lower bound of the Cauchy--Davenport theorem?

Our main results show that, rather surprisingly, in many circumstances it is possible to
obtain not merely an asymptotic version of the usual sumset bound, but even the exact bound
itself. For example, in ${\mathbb Z}$, we show that if $A$ and $B$ have size $n$ then there are
three elements $b_1,b_2,b_3 \in B$ such that $|(A+b_1)\cup (A+b_2) \cup (A+b_3)|\ge 2n-1$.
And for ${\mathbb Z}_p$ itself, we show the following: if $A$ and $B$ have size $n$, where
$n \leq p/3$, then there is a subset $B'$ of $B$ of size $c$ such that $|A+B'| \geq 2n-1$. Here
$c$ is an absolute constant. In the inverse direction for this result, we show
that if for
every subset $B'$ of $B$ of size $c$ we have
$|A + B'| \leq 2n-1+r$,
where $r \leq \varepsilon n$ for some absolute constant $\varepsilon$, then $B$ is contained in an
arithmetic progression of size
$n+r$. We also prove `unbalanced' forms
of our results, when the sizes of $A$ and $B$ may differ.

As an application, we prove some
considerable extensions of the Erd\H{o}s-Heilbronn problem. We also present versions in the continuous setting, and give
several open problems.

\end{abstract}

\maketitle

\section{Introduction}\label{intro}

The aim of this paper is to introduce a new direction in the study of sumset sizes, by asking what
happens when we may only use a bounded number of terms from one of the sets. Our main interest is in
sumsets in ${\mathbb Z}_p$, where, as usual, $p$ is a prime number; in fact, throughout
this paper, $p$ will always stand for an arbitrary prime. We start with some
background.

\vspace{5pt}
Cauchy~\cite{Cauchy} was the first to study sums of subsets in ${\mathbb Z}_p$;
over one hundred years later, his result was rediscovered by
Davenport~\cite{Dav1, Dav2}. Their result asserts that if $\emptyset \ne A, B \subset {\mathbb Z}_p$ and
$|A|+B|\le p+1$ then
\begin{equation}\label{BM1}
|A+B|\ge |A|+|B|-1.
\end{equation}
Proving this for sets of integers, without any restrictions on their sizes, is entirely trivial,
and may be viewed as an analogue of the one-dimensional case of the Brunn--Minkowski inequality.

\vspace{5pt}
The Cauchy--Davenport theorem was followed by important
contributions concerning sums of subsets of groups, including ${\mathbb Z}$ itself, by
Mann~\cite{Mann1, Mann2}, Kneser~\cite{Knes}, Vosper~\cite{Vosp1, Vosp2}, Erd\H{o}s and
Heilbronn~\cite{ErdHeil}, Freiman~\cite{Frei-59, Frei-62, Frei-book}, Pl\"unnecke~\cite{Plu-70},
Ruzsa~\cite{Ruz-89}, and others until the 1990s,
when the subject really took off (see e.g. \cite{BSZ, BaBo, GyarMatRu, LevSme, Petr, Ruz-triangle-96, SerZem,
Stan, tao}). For various discrete analogues of the Brunn--Minkowski inequality, see
\cite{BoLe, BoLe-96, GarGro, GonHen, GreenTao, HerIgYep, IgYepZva}  and many other results;
for surveys of the inequality itself, see \cite{Bar, Gard}.

\vspace{5pt}
One should distinguish between `direct' results, giving lower bounds for subset sums, and `inverse'
results, that characterise the cases when the subset sum is close to its minimum (where `close' may be interpreted in
several different ways). Starting with Vosper~\cite{Vosp1}  and Freiman~\cite{Frei-59, Frei-62, Frei-book, Frei-87}, much research has been done on inverse results: see Nathanson~\cite{Nath-first-inverse, Nath-inverse, Nathbook}, Bilu, Lev and Ruzsa~\cite{biluzp}, Bilu~\cite{Bilu}, Nathanson and Tennenbaum~\cite{NathTen}, Breuillard, Green and Tao~\cite{BGT-doubling} and Serra and Z\'emor~\cite{SerZem}, among others.

\vspace{5pt}
Considerable attention has also been given to `restricted' sumsets, meaning sumsets in which
we do not consider the sum of every pair. The best--known example of this is
the Erd\H{o}s--Heilbronn conjecture~\cite{ErdHeil}, proved by
Dias da Silva and Hamidoune~\cite{DiHa}, and then by Alon, Nathanson and Ruzsa~\cite{AlNaRu}. Restricted sumsets sit at the core
of the Balog--Szemer\'edi--Gowers theorem~\cite{BSG1,BSG2}, which is a fundamental tool in additive combinatorics. This notion
has also been explored by several authors including Freiman, Low and Pitman~\cite{FreLowPit99}, Lev~\cite{Lev-00a, Lev-00b, Lev-01}, K\'arolyi~\cite{Kar-05}, Vu and Wood~\cite{VuWood-09}, Tao~\cite{tao}, Griesmer~\cite{griesmers} and Shao~\cite{shao}.

\vspace{5pt}
In this paper we wish to find out what happens if in \eqref{BM1} or its analogues
we are allowed to use only boundedly many elements of $B$. This question is of
interest not only in ${\mathbb Z}_p$, but also in additive groups and even in ${\mathbb Z}$.

\vspace{5pt}
In the rest of this section we describe our main results.

\vspace{10pt}
\subsection{Direct results in $\mathbb{Z}$}
We start in the simplest place, ${\mathbb Z}$ itself. If we are summing just one set $A$ with
itself, then of course with $a_1 = \min A$ and $a_2 = \max A$ we have that the sets $A + a_1$ and
$A + a_2$ meet only at $a_1+a_2$, so that $|A+\{a_1,a_2\}| = 2 |A| - 1$. But with
a sumset $A+B$ of two sets then it is no longer the case that for some two elements $b_1,b_2$ of
$B$ we have that $|A+\{b_1,b_2\}| = |A|+|B| - 1$, even when the sets are the same size.
For example, taking $n$ to be a multiple of $3$, if $A$ is the set $[1,2n/3] \cup [n,4n/3]$ then it is easy to see that
any translate of $A$ by a distance $d$, where $0 \leq d \leq n$, meets $A$ in at least $n/3$
elements. Hence if $B$ is the interval $[0,n]$ then $A$ and $B$ have size $n+1$ and for any $b_1,b_2 \in B$ we have that
$|A+\{b_1,b_2\}| \leq |A|+|B|-n/3$.

\vspace{5pt}
In light of this simple example, it is very surprising that, if we move up to {\em three}
elements of $B$, then in fact we can recover the exact lower bound in \eqref{BM1}. Of course, we
will be taking $|A| \geq |B|$, since if $B$ is much larger than $A$ then the lower bound in
\eqref{BM1} is greater than $3 |A|$, which is the most one could ever obtain by taking the union of
three translates of $A$. This is our first result.

\begin{theorem}\label{intro_three_translates_Z}
Let $A$ and $B$ be finite non-empty subsets of $\mathbb{Z}$ with $|A| \ge |B|$. Then there exist elements
$b_1,b_2,b_3 \in B$ such that
$$|A+\{b_1,b_2,b_3\}| \ge |A|+|B|-1.$$
\end{theorem}

\vspace{5pt}
Actually, if $B$ is allowed to be {\em larger} than $A$, but only by a given ratio, then
a bounded number of elements of $B$ does still suffice to obtain the bounds in \eqref{BM1}.\\

\begin{theorem}\label{intro_few_translates_Z}
For every $\alpha >0$ there exists $c$ such that, whenever $A$ and $B$ are finite non-empty subsets of
$\mathbb{Z}$ with $|A| \geq \alpha |B|$, there exist elements $b_1, b_2,
\hdots, b_c \in B$ such that
\[
\large|A+\{b_1, \hdots, b_c\}\large| \ge |A|+|B|-1.
\]
\end{theorem}

\vspace{10pt}
\subsection{Direct results in $\mathbb{Z}_p$}
Let us now turn to the deeper question of what happens in the cyclic group ${\mathbb Z}_p$ of prime order. Can we again
recover the exact bounds, this time in the Cauchy--Davenport theorem, if we insist that only
a given number of terms of $B$ may be used? The same examples as in $\mathbb Z$ show that two
translates do not suffice. However, a bounded number of translates is enough. This is again very
surprising, and is the one of the main results of our paper. In a certain sense, it is going far
beyond results like the Erd\H os--Heilbronn conjecture, where we restrict the allowed sums only by
forbidding $a$ to equal $b$ in a sum $a+b$ we use: here on the contrary we only {\em allow} a fixed number
of members of $B$ to appear in the sums.

\begin{theorem}\label{translates_of_A_Zp_equality}
There exists a universal constant $c$ such that the following holds. Whenever
$A$ and $B$ are subsets of $\mathbb{Z}_p$ with $|A|=|B|\leq p/3$,
there exist $b_1, \hdots, b_c \in B$ such that $$|A+\{b_1, \hdots, b_c\} |\geq |A|+|B|-1.$$
\end{theorem}

\vspace{5pt}
It would be fascinating to know if the number $c$ of summands can be taken to be 3, if the sizes of $A$ and $B$ are a
sufficiently small multiple of $p$.

\vspace{5pt}
More generally, we may allow the sizes of $A$ and $B$ to approach $p/2$, and also their sizes need
not be the same. Note that we {\em have} to let the number of allowed summands in $B$ increase as
the sum of the sizes of $A$ and
$B$ gets closer to $p$: this is easily seen if we choose $A$ to be a random subset.
Indeed, if $A$ is a random subset of $\mathbb{Z}_p$ of size $(1/2 - \beta) p$ and $c$ is fixed then
with high probability we have that for {\em any} $b_1, \hdots, b_c$ the union
$\cup_{1\leq i\leq c}(A+b_i)$ has size about $(1-(1/2 + \beta)^c))p$, and this is smaller than
$(1-3\beta)p$ if $\beta$ is small. In other words, for such an $A$, and for any $B$ at all of size
$(1/2 - \beta) p$, we have that there do not exist $c$ points of $B$ whose sumset with $A$ has size
even close to $|A| + |B|$.

\begin{theorem}\label{translates_of_A_Zp}
For all $\alpha, \beta >0$ there exists a constant $c$ such that the
following holds. Whenever $A$ and $B$ are subsets of $\mathbb{Z}_p$ with
$\alpha |B| \leq |A| \leq \frac{1}{\alpha}|B|$ and $|A|+|B| \leq (1-\beta)p$, there
exist $b_1, \hdots, b_c \in B$ such that
\[
|A+\{b_1, \hdots, b_c\} |\geq |A|+|B|-1.
\]
\end{theorem}

\vspace{5pt}
Interestingly, if $A$ is larger than a fixed (sufficiently large) multiple
of the size of $B$ then we do have that
{\em three} translates are enough.\\

\begin{theorem}\label{intro_general_three_translates_Zp}
For every $\beta>0$ there exists $\alpha>0$ such that the following holds.
Whenever $A$ and $B$ are non-empty subsets of
$\mathbb{Z}_p$ with $|B| \leq \alpha |A|$ and $|A|+|B| \leq (1-\beta)p$,
there exist
elements $b_1,b_2,b_3 \in B$ such that
$$|A+\{b_1, b_2,b_3\}| \ge |A|+|B|-1.$$
\end{theorem}

(We mention in passing that this result actually means that, in Theorem~\ref{translates_of_A_Zp}, the condition
that $|A| \leq \frac{1}{\alpha}|B|$ may be removed.)

\vspace{5pt}
An important case of this theorem is when $B$ is an interval, i.e. a set of the form $[x,y]= \{x, x+1, \dots , y\} \subset {\mathbb Z}_p$.
One is tempted to imagine that in this situation everything is much easier, but this is not the case,
and indeed a substantial part of our approach to the
various results in the paper consists of first proving the result when $B$ is close to an interval
(which is usually a large part of the overall work) and then seeing how, if at all, the proof
can be modified for general $B$. When $B$ is an interval we also recover the minimal possible number of
translates of $A$.

\begin{theorem}\label{intro_four_translates_Zp}
Let $A$ and $B$ be non-empty subsets of $\mathbb{Z}_p$ with
$ |B| \leq |A| \leq 2^{-20}p$ and $B$ an interval.
Then there exist elements $b_1,b_2,b_3 \in B$
such that $$|A+\{b_1, b_2,b_3\}| \ge |A|+|B|-1.$$
\end{theorem}

\vspace{10pt}
\subsection{Inverse results in $\mathbb{Z}_p$}
We also prove inverse theorems for our results in ${\mathbb Z}_p$ -- where by `inverse' we as usual
mean statements of the form `what happens if the inequality is close to being tight'. (These are also often
known as `stability' results.) These show that, unless the
lower bounds hold `with room to spare', our sets must be highly structured, and in fact be very close to
arithmetic progressions.

One of our main results is the following, which may be viewed
as an extension of Freiman's $3k-4$ theorem in this group where we demand that only a bounded number of
summands from $B$ are used.

\begin{theorem}\label{stability_in_Zp_for_intro}
  For all $\alpha,\beta>0$ there exist constants $c$ and $\varepsilon>0$ such that the following holds.
  Let $A$ and $B$ be subsets of $\mathbb{Z}_p$ of size at least $2$,
  with $\alpha |B| \leq |A| \leq
  \alpha^{-1}  |B|$ and
  $|A|+|B|\leq (1-\beta)p$. Suppose that for any $c$ elements
  $b_1, \hdots, b_c \in B$ we have
$$|A+\{b_1, \hdots, b_c\}|\leq |A|+|B|-1+r$$
where $r \leq \varepsilon |B|$. Then $B$ is contained in an arithmetic progression of size
$|B|+r$.
\end{theorem}

\vspace{5pt}
This is actually the key to proving several of our {\em direct} results in ${\Bbb Z}_p$. In fact, our proof shows that $\varepsilon$ depends only on $\beta$: it is independent of $\alpha$.

We remark
that, in the form stated, if $A \neq B$ then this is not really a `true' inverse result to our results
above, as it does not describe
the structure of $A$. In fact, we cannot hope to insist that $A$ is contained in a short arithmetic progression, since
we may always add to $A$ a small `sprinkling' of faraway points without affecting the relevant properties of the
sumsets with a bounded number of points from $B$. But one {\it can} obtain that $A$ has small symmetric difference
with some short arithmetic progression -- we will say a few words about this when we come to prove
Theorem~\ref{stability_in_Zp_for_intro}.

\vspace{10pt}
\subsection{Abelian groups}
There is an important tool in many of our results, that will
often allow us to make progress with `arbitrary' sets. This may be of independent
interest.

\begin{theorem}\label{implication_shao1-prov}
For all $K$ and $\varepsilon >0$ there is an integer $c$ such that the following holds.
Let $A$ and $B$ be finite subsets of an abelian group. Then there are subsets $A^* \subset A$
and $B^* \subset B$, with $|A^*|\ge (1-\varepsilon)|A|$ and
$|B^*|\ge (1-\varepsilon)|B|$, such that if we select points $b_1,\hdots,b_c$ uniformly at random from
$B$ then
\[
{\mathbb E}_{b_1,  \dots , b_c\in B}\ |A+\{b_1, \dots , b_c\}|
\geq \min\big((1-\varepsilon)|A^*+ B^*|,\ K |A|,\ K |B| \big).
\]
\end{theorem}

\vspace{5pt}
One could view Theorem~\ref{implication_shao1-prov} as an approximate form of the sharp results we wish to prove in
${\mathbb Z}$ and ${\mathbb Z}_p$.
We remark that simple examples show that
the dependence of $c$ in terms of $\varepsilon$ and $K$ is necessary.

\vspace{5pt}
Amusingly, one can use a variant of Theorem~\ref{implication_shao1-prov} (namely
Theorem~$8'$ below)
to deduce Roth's theorem~\cite{Roth}
on three-term arithmetic
progressions. This is not surprising, since (as we shall explain later) our result builds
on work of Shao~\cite{shao}, which itself builds on the arithmetic regularity lemma of
Green~\cite{green} -- and the
arithmetic regularity lemma has Roth's theorem as an immediate consequence. But it is interesting that
the deduction from Theorem~$8'$ is direct.

\vspace{10pt}
\subsection{Restricted Sums}
Although our aim in this paper is not so much to give applications, we do expect that our results will
prove to be useful tools. As an example, as an application of our methods, we turn our attention to the Erd\H{o}s--Heilbronn problem.
Our results turn out to yield, very easily, some considerable extensions of this. For example, we consider
`restricted' sums, where for each $a \in A$ there are some $b \in B$ that we are
not allowed to use when forming the `restricted sumset' of $A$ with $B$.

\begin{theorem}\label{EH_for_intro}
  For each $\beta>0$ and integer $d$ there is an $n_0$ such that following holds. Let $A$ be a subset of $\mathbb{Z}_p$ with
  $n_0 \le |A|=n < (1-\beta)p/2$, and suppose that we form all sums $a+b$, where $a,b \in A$, except that for
  each $a \in A$ there is a set of $d$ values in $A$ that we are not allowed to take as $b$. Then the resulting `restricted sumset' has
  size at least $2n-1-2d$.
\end{theorem}

\vspace{5pt}
Note that the Erd\H{o}s--Heilbronn problem corresponds to the case when $d=1$ and for each $a \in A$ the one
element of $A$ that we cannot sum with $a$ is $a$ itself.

\vspace{5pt}
In fact, we prove an extension of this where the number of summands from $A$ that we take is bounded, as in the spirit of the rest of this paper.
Indeed, this actually {\it helps} us: because we are already in the situation where the second summands are
from a bounded set, the additional constraint that certain pairs are not allowed turns out to be much easier to
handle than it would be in general. So it turns out that our results here will follow directly from
Theorem~\ref{stability_in_Zp_for_intro}.

\vspace{10pt}

\subsection{Sums in the Continuous Setting}
Our results tend to have consequences in the continuous setting, about sums of compact sets in the reals or
in the circle $\mathbb{T}=\mathbb{R}/\mathbb{Z}$.
We mention here two typical examples. In the statements we shall use $|\cdot|$ to denote Lebesgue measure on the Euclidean space $\mathbb{R}$ or Haar measure (normalized Lebesgue measure) on the circle $\mathbb{T}=\mathbb{R}/\mathbb{Z}$. As we will
see, the deductions from the corresponding discrete results (namely Theorem~\ref{intro_three_translates_Z} for the first and Theorems~\ref{translates_of_A_Zp}
and \ref{intro_general_three_translates_Zp} for the second) 
will be fairly straightforward.

\begin{corollary}\label{cont_intro_three_translates_Z}
Let $A$ and $B$ be non-empty compact subsets of $\mathbb{R}$ with $|A| \ge |B|$. Then there are elements
$b_1,b_2,b_3 \in B$ such that
$$|A+\{b_1,b_2,b_3\}| \ge |A|+|B|.$$
\end{corollary}

\begin{corollary}\label{cont_translates_of_A_Zp}
For all $\alpha, \beta >0$ there exists a constant $c$ such that the
following holds. Whenever $A$ and $B$ are non-empty compact subsets of $\mathbb{T}$ with
$\alpha |B| \leq |A|$ and $|A|+|B| \leq 1-\beta$, there
are $b_1, \hdots, b_c \in B$ such that
\[
|A+\{b_1, \hdots, b_c\} |\geq |A|+|B|.
\]
\end{corollary}

It is tempting to believe that
these corollaries are equivalent to the corresponding results in the discrete case, but this does not
appear to be the case: it does not seem easy to prove the reverse implications, essentially because the
points we sample (the $b_i$) in the continuous case may not `line up' with the way we are embedding our
discrete sets into the continuous world.

\vspace{10pt}
\subsection{Structure}
The plan of the paper is as follows. In Section~2 we mention various results from the literature that we
will make use of. In Section~3 we prove
the results in ${\mathbb Z}$, in
the regime where $A$ and $B$ have the same size. We also determine the cases of equality. Then in Section~4 we prove
Theorem \ref{implication_shao1-prov}, which is a key component in several of later proofs. Section~5
contains the
result in ${\mathbb Z}_p$ when $B$ is an interval, and also deals with the situation when $A$ is much larger than $B$, and then Section~6 deals with the
general situation in ${\mathbb Z}_p$ (including Theorem \ref{stability_in_Zp_for_intro}).
In Section 7 we give the applications to the Erd\H{o}s--Heilbronn problem mentioned above. Then Section 8
has the the continuous versions of our statements.
We finish in Section~9 with some open problems.

\vspace{5pt}
It is worth pointing out that a reader
who is only interested in what happens in the integers will still need Section 4 since, although the
proof of Theorem \ref{intro_three_translates_Z} is direct, the proof of Theorem~\ref{intro_few_translates_Z}
is much more involved and does make use of
Theorem~\ref{implication_shao1-prov}. Indeed, we will actually deduce
Theorem~\ref{intro_few_translates_Z} from Theorem~\ref{translates_of_A_Zp_equality}, which is a particular case of Theorem~\ref{translates_of_A_Zp}. The latter will turn out to follow from a form of
Theorem~\ref{stability_in_Zp_for_intro} itself.

\vspace{10pt}
\subsection{Notation}
Our notation is standard. To make our paper more readable, we often omit integer-part signs when these do not affect the
argument. For example, given
$x>0$, we use the notation $[x]$ for $\{1,2, \hdots, \lfloor x \rfloor\}$.

\vspace{5pt}
Sometimes we write `$x \mod d$' as shorthand for the infinite arithmetic progression
$\{ y \in \mathbb{Z}: y \equiv x \mod d \}$, and refer to it as a {\em fibre} mod $d$.
When $S$ is a subset of $\mathbb{Z}$ we often write $S^x$ for the intersection of this fibre with
$S$ -- when the value of $d$ is clear. (We sometimes write $S^x$ as $S^x_d$ when we want to stress the value of $d$.)
Thus $S^x=S\cap \pi^{-1}(x)$, where $\pi = \pi_d$ denotes
the natural projection from $\mathbb{Z}$ to $\mathbb{Z}_d$. We also write $\widetilde{S}$
for $\pi_d(S)$.

\vspace{5pt}
When we write a probability or an expectation over a finite set, we always assume that the elements
of the set are being sampled uniformly. Thus, for example,
for a finite set $X\subset \mathbb{Z}$ we denote the expectation and probability when we sample uniformly over
all $x \in X$ by respectively $\E_{x\in X} \text{ and } \Prob_{x\in X}.$ Similarly, an expectation when
we sample uniformly over a product space $X_1 \times \cdots \times X_n$ may be written as
$\E_{\substack{x_j \in X_j \\ j \in \{1,\hdots,n \}}}$.

We also often sample uniformly over all $c$-sets of a given set $X$. In most of those cases,
we could instead sample $c$ elements uniformly and independently, but the notation would tend to get
unwieldy, and this is why we use the sampling over all $c$-sets instead.

\vspace{5pt}
For a set $X$ and integer $c \in \mathbb{N}$ we as usual denote by $X^{(c)}$ the family of
all subsets of $X$ of size $c$, and by $X^{(\le c)}$ the family of
all subsets of size at most $c$. However, in several places it is very useful, when we want to
specify that a set is small, to also allow $X$ to have size smaller than $c$, and in that case
we make the convention that $X^{(c)}$ denotes the singleton $\{ X \}$. We hope that the reader will
not mind this convention, which will allow arguments to flow without having to deal
with many (unimportant) special cases.

\vspace{5pt}
For more general background on sumsets, or for background on any of the results
mentioned in the next section, see the survey of Breuillard, Green and Tao \cite{BGT-doubling} or the
books of Nathanson \cite{Nathbook} or Tao and Vu \cite{taobook}.

\vspace{5pt}
To end this Introduction, we mention two further questions of a seemingly similar flavour,
although really there is little connection with our topic.
First, given sets $A$ and $B$ in an abelian group, is there a large set
of disjoint translates $A+b$ with $b \in B$? The answer is given by a simple
yet fundamental tool in additive combinatorics, Ruzsa's covering lemma~\cite{RuzsaCov}
(see also Lemma 2.14 in Tao and Vu~\cite{taobook}), which states that there is a set $X\subset B$
such that $X+A-A$ contains $B$ and $|X+A|=|A||X|$, i.e. the translates of $A$ through
elements of $X$ are disjoint and the translates of $A-A$ through $X$ cover $B$. In
particular, this shows that there are at least $|B|/|A-A|$ disjoint
translates of $A+b$ with $b \in B$.

\vspace{5pt}
Second, one might also wonder about the following related question. Given sets $A$
and $B$ in an Abelian group, are there {\em small} subsets $A'\subset A$ and
$B'\subset B$ such that $A+B=(A+B') \cup (B+A')$? A result of Ellenberg~\cite{Ell} answers
this question in ${\mathbb Z}_p^n$, showing that we can always take $|A'|, |B'| \leq c_q^n$
where $c_q<q$ (and in fact $c_q \leq \alpha q$ for some absolute constant $\alpha<1$).  Of course, this result and our main result consider very
different ranges. In \cite{Ell} the union $(A+B') \cup (B+A')$ is required to be much
bigger, potentially of size $|A||B|$, but to achieve this $A'$ and $B'$ are allowed to be very large, of
size $c_q^n$.

\vspace{5pt}
\section{Prerequisites}
As mentioned above, we will make use of
the following recent result of Shao~\cite{shao}, which can be viewed as an `almost all'
version of the Balog--Szemer\'edi--Gowers theorem~\cite{BSG1,BSG2}. The main ingredient
in this result is an application of the arithmetic removal lemma (an adaptation
of the graph removal lemma to groups) of
Green (see \cite{green}). We shall actually use a version of this result for unequal size sets
that can be proved in the same way. Here as usual we write $A+_{\Gamma}B$, where $\Gamma$ is a subset
of $A \times B$, to denote the set of all sums $a+b$ where $(a,b) \in \Gamma$. (And later we will also
use $A-_{\Gamma}B$ to denote the set of all sums $a-b$ with $(a,b) \in \Gamma$.) Here is the (variant of the)
result of Shao~\cite{shao} that we shall need.

\begin{theorem}\label{thm_shao}
For all $\varepsilon, K>0$ there exist $\delta>0$ such that the following holds. Let $G$ be an
abelian group and let $N \in \mathbb{N}$. Let $A,B \subset G$
be two subsets with $|A|, |B| \geq N$, and let $\Gamma \subset A \times B$ be a subset with $
|\Gamma|\geq (1-\delta)|A||B|$. If $|A+_{\Gamma}B| \leq KN$, then there
exist $A_0 \subset A, B_0 \subset B$ such that $$|A_0| \geq (1-\varepsilon)|A|
\text{ and } |B_0|\geq (1-\varepsilon)|B| \text{ and } |A_0+B_0| \leq |A+_{\Gamma}B|+\varepsilon N.$$
\end{theorem}

\vspace{5pt}
We will need the extension of Freiman's $3k-4$ theorem \cite{Frei-59,Frei-62} by Lev and
Smeliansky~\cite{LevSme} and Stanchescu~\cite{Stan}.

\begin{theorem}\label{3k-4}
Let $A$ and $B$ be finite non-empty subsets of ${\mathbb Z}$ with $|A+B|=|A|+|B|-1+r$, where
$r \leq \min(|A|,|B|)-3$. Then there are arithmetic progressions $P_A \supset A$ and $P_B \supset B$,
having the same common difference, such that $|P_A \setminus A|,|P_B \setminus B| \leq r$.
\end{theorem}

\vspace{5pt}
We shall also make use of the following analogue of Freiman's $3k-4$ theorem in $\mathbb{Z}_p$
(see Theorem 21.8 of Grynkiewicz~\cite{grynkiewicz-2}).

\begin{theorem}\label{3k-4_Z_p}
There is an absolute constant $\eta >0$ such that the following holds. Let $A$ and $B$ be finite nonempty subsets of $\mathbb{Z}_p$ and let $C=-(A+B)^c$ and let $r$ be a integer with $0\leq r \leq \eta p-2$. Suppose that $$|A+B|=|A|+|B|-1+r$$
and $$|A|\geq r+3 \text{, } |B|\geq r+3 \text{ and } |C|\geq r+3.$$
Then there exist arithmetic progressions $P_A$, $P_B$ and $P_c$ with the same common difference containing $A$, $B$ and $C$, respectively, such that
$$|P_A\setminus A|\leq r \text{, } |P_B \setminus B|\leq r \text{ and } |P_C \setminus C|\leq r.$$
\end{theorem}

\vspace{5pt}
Finally, we recall Vosper's theorem  \cite{Vosp1, Vosp2}.
\begin{theorem}\label{vosper}
Let $A$ and $B$ be finite subsets of $\mathbb{Z}_p$ with at least two elements such that $|A+B| = |A|+|B|-1 \leq p-2$. Then $A$ and $B$ are arithmetic progressions with the same common difference.
\end{theorem}

\section{Sums in ${\mathbb Z}$}

\vspace{7pt}
\noindent
We start by proving Theorem \ref{intro_three_translates_Z} that {\em if $A$ and $B$ are finite non-empty sets of
integers with $|A| \ge |B|$, then there exist elements
$b_1,b_2,b_3 \in B$ such that $|A+\{b_1, b_2,b_3\}| \ge |A|+|B|-1.$}

\vspace{5pt}
One might imagine that a random argument cannot
be helpful in proving this, since in the case when $A$ and $B$ are intervals then two of our elements {\em must} be the
first and last elements of $B$. But, remarkably, what we will find is that if we fix those two elements and select only the third element at random, we do obtain exactly the bound we want.

\vspace{5pt}
\noindent
{\em Proof of Theorem \ref{intro_three_translates_Z}.}
Let $B$ have first element $0$ and last element $m$, and write $\pi_m: {\mathbb Z} \to {\mathbb Z}_m={\mathbb Z}/m {\mathbb Z}$ for the canonical projection of ${\mathbb Z}$ onto ${\mathbb Z}_m$. Then $A+\{0, m\}=A\cup (A+m)$ satisfies
$\pi_m(A)=\pi_m\big(A\cup (A+m)\big)= \widetilde{A}$. Although it is easy to bound the size of this set $A\cup (A+m)$, we record it as a lemma because we shall make use of it several times later on.

\begin{lemma}\label{lem8.1}
We have
\[
\bigg|A\cup (A+m) \bigg| \geq |A|+|\widetilde{A}|.
\]
\end{lemma}
\begin{proof}
Note that $(A+m)^x=A^x+m$ for every $x\in {\mathbb Z}$. Hence, if $x\in \widetilde{A}=\pi_m(A)$ then
$(A+m)^x=A^x+m\ne A^x$,
so
$
\bigg| \big[A\cup (A+m)\big]^x\bigg|=\bigg|A^x\cup (A+m)^x \bigg| \geq |A^x|+1.
$
Consequently,
\[
\bigg|A\cup (A+m)\bigg| \geq \sum_{x\in \widetilde{A}} \bigg|[A\cup (A+m)]^x\bigg| \geq \sum_{x\in \widetilde{A}} |A^x|+1 = |A|+|\widetilde{A}|,
\]
completing the proof.
\end{proof}

\vspace{5pt}
We now turn to the contribution from the `random' term. It does seem very fortunate that this meshes so well with the bound above. Again, we record it as a separate lemma for future use.

\begin{lemma}\label{lem8.2specialcase} If we choose an element $b$ from $B \setminus \{m\}$, uniformly at random, then
$$\E_{b\in B\setminus \{m\}}\bigg|(A+b) \setminus \pi_m^{-1}(\widetilde{A})\bigg| \geq
|A| \max\bigg(0,\frac{|\widetilde{B}|-|\widetilde{A}|}{|\widetilde{B}|}\bigg).$$
\end{lemma}
\begin{proof}
Clearly, $\pi_m: B\setminus\{m\} \to \widetilde{B}$ is a bijection. Also, for $a\in A$  we have

\begin{equation*}
        \Prob_{b \in B\setminus \{m\}} \bigg(a+b\not \in \pi^{-1}_m(\widetilde{A})\bigg) =\Prob_{\widetilde{b}\in \widetilde{B}}\bigg(\widetilde{a}+\widetilde{b}\not \in \widetilde{A}\bigg)
        \geq \max\bigg(0,\frac{|\widetilde{B}| - |\widetilde{A}|}{|\widetilde{B}|}\bigg).
\end{equation*}
so
\begin{equation*}
        \E_{b \in B\setminus \{m\}} \bigg|(A+b)\setminus \pi^{-1}_m(\widetilde{A})\bigg| = \sum_{a\in A}\Prob_{b\in B\setminus\{m\}}\bigg(a+b\not \in \pi_m^{-1}(\widetilde{A})\bigg)
        \geq |A|\max\bigg(0,\frac{|\widetilde{B}| - |\widetilde{A}|}{|\widetilde{B}|}\bigg),
\end{equation*}
as claimed.
\end{proof}

\vspace{5pt}
To prove Theorem~\ref{intro_three_translates_Z}, we combine these two lemmas. We have

\begin{eqnarray*}
\E_{b \in B\setminus \{m\}} \bigg|A
+\{0,b,m\}\bigg|&\geq&|A|+|\widetilde{A}| +|A| \max\bigg(0,\frac{|\widetilde{B}| - |\widetilde{A}|}{|\widetilde{B}|}\bigg)\\
&=&|A|+|\pi_m(A)|+|A|\max\bigg(0, \frac{|B|-1-|\pi_m(A)|}{|B|-1} \bigg).
\end{eqnarray*}
If $|\pi_m(A)| \geq |B|-1$ then
\begin{equation*}
    \E_{b\in B \setminus \{m\}}\bigg|A+\{0,b,m\}\bigg| \geq |A|+|\pi_m(A)|
    \geq |A|+|B|-1.
\end{equation*}
On the other hand, if $|\pi_m(A)| \leq |B|-1$ then
\begin{eqnarray*}
    \E_{b\in B \setminus \{m\}}\bigg|A+\{0,b,m\}\bigg| \hspace{-17pt} &&\geq |A|+|\pi_m(A)|+ |A|\frac{|B|-1-|\pi_m(A)|}{|B|-1}\\
    \hspace{-17pt}&&\geq |A|+|\pi_m(A)|+|B|-1-|\pi_m(A)|
    =|A|+|B|-1,
\end{eqnarray*}
completing our proof of Theorem~\ref{intro_three_translates_Z}. \hfill{$\square$}

\vspace{5pt}
In fact, the proof of Theorem~\ref{intro_three_translates_Z} above shows the following stronger result, which we record here. Let $A$ and $B$ be finite non-empty subsets of integers with $|A|\geq |B|$ and $B$ having smallest element $0$ and greatest element $m$.  Then we have
\begin{equation}\label{thm1.0.0}
     \max_{b_1,b_2,b_3 \in B}\bigg|A+\{b_1,b_2,b_3\}\bigg| \geq \E_{b\in B\setminus \{m\}} \bigg| A +\{0,b,m\} \bigg| \geq |A|+|B|-1.
\end{equation}

\vspace{7pt}
We now collect together some variants of the results above that we shall need at later points in the paper. We urge the reader to skip these (rather pedestrian) variations, returning to them only when they are actually needed in the sequel. To avoid clutter and some repetition, we leave the proofs for the Appendix.

First, we investigate when equality is attained in inequality \eqref{thm1.0.0}, that is when $|A+\{b_1,b_2,b_3\}| \leq |A|+|B|-1$ for all $b_1,b_2,b_3 \in B$.

\begin{theorem}\label{equality_four_translates}
  Let $A$ and $B$ be finite non-empty subsets of $\mathbb{Z}$ such that $|A| = |B|$, with $\min B = 0$ and
  $\max B = m$. Suppose that when we choose an element $b$ of $B \setminus \{0,m\}$ uniformly at random we have
\begin{equation}\label{thm23.0.0}
    \E_{b \in B \setminus \{0,m\}} \bigg|A+\{0,b,m\}\bigg| \leq |A|+|B|-1.
\end{equation}
Then $A$ and $B$ are arithmetic progressions with the same common difference.

\vspace{7pt}
In particular, suppose that when we choose any three elements $b_1, b_2, b_3$ of $B$ we have
\[
\bigg|A+\{b_1,b_2,b_3\}\bigg| \le |A|+|B|-1.
\]
Then $A$ and $B$ are arithmetic progressions with the same common difference.
\end{theorem}

Next we have a version of Theorem \ref{intro_three_translates_Z}.

\begin{theorem}\label{technical_three_translates_Z}
Let $A$ and $B$ be finite non-empty subsets of $\mathbb{Z}$, with $\min B  = 0$ and $\max B =m$. Then
$$\E_{b\in B \setminus \{m\}}\bigg|A+\{0,b,m\} \bigg| \ge |A|+|\pi_m(A)|+|A|\max\bigg(0, \frac{|B|-1-|\pi_m(A)|}{|B|-1} \bigg).$$
In particular, if $|A|\geq |B|-1$ we have
$$\E_{b\in B \setminus \{m\}}\bigg|A+\{0,b,m\}\bigg| \ge |A|+|B|-1$$
and if $|A|\leq |B|-1$ we have
$$\E_{b\in B \setminus \{m\}}\bigg|A+\{0,b,m\}\bigg| \ge 2|A| + |\pi_m(A)| \frac{|B|-1-|A|}{|B|-1}.$$
\end{theorem}

We also have a simple variant of Lemma~\ref{lem8.2specialcase}.

\begin{lemma}\label{lem8.2} Let $A$, $B$ and $A_1$ be finite non-empty subsets of $\mathbb{Z}$, with
  $\min B =0$ and $\max B =m$. Then with $\widetilde{A}=\pi_m(A)$ and $\widetilde{B}=\pi_m(B)$ we have
\[
\ \ \ \ \ \ \ \ \ \ \E_{b\in B \setminus \{m\}} \bigg|(A_1+b) \setminus \pi_m^{-1}(\widetilde{A})\bigg| \geq
|A_1| \max\bigg(0,\frac{|\widetilde{B}|-|\widetilde{A}|}{|\widetilde{B}|}\bigg).
\ \ \ \ \ \ \ \ \ \ \square
\]
\end{lemma}

\vspace{7pt}
Our next lemma follows from Lemma~\ref{lem8.2} by simple manipulations.

\begin{lemma}\label{residue.vs2}
Let $A$ and $B$ be finite non-empty subsets of $\mathbb{Z}$. Let $\min(B)=0$ and $\max(B)=m$, and suppose that $|B|-1=|\widetilde{B}| \geq 8|\widetilde{A}|$, where $\widetilde{A}=\pi_m(A)$ and $\widetilde{B}=\pi_m(B)$. Then
\[
  \ \ \ \ \ \ \ \ \ \ \ \ \ \E_{b_2,b_3 \in B\setminus\{m\}} \bigg|A+\{0,b_2,b_3\}  \bigg|\geq 2.5|A|.
  \ \ \ \ \ \ \ \ \ \ \ \ \ \ \square
\]
\end{lemma}

Finally, we have a strengthening of Theorem~\ref{intro_three_translates_Z}.

\begin{theorem}\label{strengthening_of_three_translates}
  Let $A$ and $B$ be finite non-empty subsets of $\mathbb{Z}$
with $|A| = |B|$, with $\min B  = 0$ and $\max B =m$. Then
$$\E_{b\in B \setminus \{m\}}\bigg|A+\{0,b,m\}\bigg| \ge |A|+|B|-1+\max\bigg(0,\frac{(2|\pi_m(A)|-m)(
m-(|B|-1))-1}{|B|-1} \bigg).$$
\end{theorem}

All of the above results are proved in the Appendix.

\vspace{5pt}

To end this section, we return to {\em two} translates. As pointed out above, even if $|A|=|B|$, there may
not be two points in $B$ whose sum with $A$ gives at least $|A|+|B|-1$ points. When {\it would} two
translates suffice?
In the first version of this paper we proved the following result.
\begin{theorem}\label{two-translates}
There is a constant $K>0$ such that for any subsets $A$ and $B$ of $\mathbb{Z}$ with $|B|=n \geq 2$ and
$|A| \ge Kn\log n$ there exist elements $b_1,b_2$ in $B$
such that
\[
|A+\{b_1,b_2\}| \geq |A|+|B|-1.
\]
\end{theorem}
One of the referees was kind enough to point out that this result is an immediate
consequence of a theorem of Konyagin and Lev~\cite{KoLe} about the Erd\H{o}s--Heilbronn--Olson
function $f_A(s)=|(A+s)\setminus A|$. Writing $f_A(S)$ for $\max \{f_A(s): s\in S\}$,
Konyagin and Lev proved that {\em there is an absolute constant $c>0$ such that $f_A(S)\ge |S|$
holds for all finite sets $A\subset {\mathbb Z}$, $S\subset {\mathbb N}$ with $|A|>1$ and $|S|<c|A|/\log |A|$.}

\vspace{5pt}
Now, trivially, $|A+\{b_1, b_2\}|=|A|+f_A(b_2-b_1)$, so
\[
\max_{b_1, b_2 \in B} |A+\{b_1, b_2\}| = |A|+f_A(S),
\]
where $S=\{b_2-b_1: b_1, b_2 \in B, b_1<b_2\}$. Since $|S|\ge n-1$, Theorem~\ref{two-translates}
follows.

\vspace{5pt}
In fact, Lev~\cite{Levtwo} and Huicochea~\cite{Hui} proved analogues of the Konyagin--Lev theorem
for ${\mathbb Z}_p$, which imply the analogue of Theorem~\ref{two-translates} in ${\mathbb Z}_p$.
We are grateful to the referee for drawing our attention to these results of Konyagin, Lev and
Huicochea.

\vspace{5pt}

\section{Proof of Theorem \ref{implication_shao1-prov}}

Our aim in this section is to prove Theorem \ref{implication_shao1-prov}. The main
ingredient of the proof is Theorem \ref{thm_shao}, and our strategy is to
apply it several times in succession. In fact, we shall prove Theorem \ref{implication_shao1-prov}
in the following slightly stronger form; the easy deduction of Theorem \ref{implication_shao1-prov}
is by changing the values of $\varepsilon$ and $c$ slightly, as required.

\vspace{7pt}

\noindent {\bf Theorem $8'$.} {\em For all $K$ and $\varepsilon >0$ there is an integer $c$ such that the following holds.
Let $A$ and $B$ be finite subsets of an abelian group. Then there are subsets $A^* \subset A$
and $B^* \subset B$, with $|A^*|\ge (1-\varepsilon)|A|$ and
$|B^*|\ge (1-\varepsilon)|B|$, such that if we select points $b_1,\hdots,b_c$ uniformly at random from
$B^*$ then
\[
\hspace{51pt}      \E_{b_1, \hdots, b_c \in B^*} |A^*+ \{b_1,\hdots,b_c\}|
 \geq \min\big( (1-\varepsilon)|A^*+ B^*|,\ K |A^*|,\ K |B^*| \big). \hspace{60pt}
\]
}

\vspace{-5pt}
\noindent
Note that Theorem $8'$ asserts the existence of a pair of large subsets $A^*$ and $B^*$ with a property that is
`intrinsic' to them.

\vspace{5pt}
\noindent
{\em Proof of Theorem~$8'$.}
Fix $\varepsilon>0$  and $K>0$, where we assume that $\varepsilon$ is sufficiently small and
$K$ is sufficiently large.
Pick $s=\lfloor \frac{50K}{\varepsilon} \rfloor$. Let
$\delta$ be given by Theorem \ref{thm_shao} with parameters $\frac{\varepsilon}{s}, 10K$.
Also pick  $\alpha$ sufficiently small that
$(1-t)^{\frac{10K}{\delta}}\leq 1-\frac{10K}{2\delta}t$ for $t \leq \alpha$. Finally pick $c\geq \frac{10K}{\delta}$ sufficiently large that
$(1-(1-\alpha)^n) \geq \frac{1-\varepsilon}{1-\varepsilon/2} \geq \frac{1}{2}$. Let  $n_A$ and $n_B$ be the sizes of $A$ and $B$ respectively.

We shall examine a process in which we repeatedly apply Theorem \ref{thm_shao} in order to construct a decreasing sequence of $s+1$ pairs of sets $(A,B)=(A_0,B_0),(A_1,B_1), \hdots ,(A_s,B_s),$
satisfying $A_i \subset A_{i-1} \text{ and } B_i \subset B_{i-1}$ and
$|A_i|\geq (1-\varepsilon/s)^i n_A$ and  $|B_i|\geq (1-\varepsilon/s)^i n_B $.
Fix $i<s$ and assume that the pair of sets $(A_i,B_i)$ has already been
constructed. We shall either stop the process at step $i$ or construct the pair of
sets $(A_{i+1},B_{i+1})$.

Let $A_i+B_i=C_i^+\sqcup C_i^-$ be the partition into `popular'  and `unpopular' elements  given by $ C_i^+=\{c \in A_i+B_i: \
|(c-A_i)\cap B_i| \geq \alpha |B_i|\}$, and
$C_i^-=\{c \in A_i+B_i: \ |(c-A_i)\cap B_i| < \alpha |B_i|\}$.
Also, let the partition $A_i \times B_i=\Gamma_i\sqcup \Gamma_i^c$ be given by $\Gamma_i=\{(a,b) \in A_i \times
B_i: \ a+b \in C_i^+  \} \subset A_i \times B_i$, and $\Gamma_i^c=\{(a,b) \in A_i
\times B_i : \ a+b \in C_i^-  \} \subset A_i \times B_i$,
so that $A_i+_{\Gamma_i}B_i=C_i^+$  and  $A_i+_{\Gamma^c_i}B_i=C_i^-$.
Finally, for
each $x \in A_i+B_i$ set
\[
r_i(x)=|(x-A_i)\cap B_i|=|\{(a,b)\in A_i \times B_i: \
x=a+b\}|,
\]
so that $ \sum_x r_i(x)=|A_i||B_i|.$
We stop this process `early', at step $i$, if $$|\Gamma_i|<(1-\delta)|A_i||B_i| \text{ or } |A_i+_{\Gamma_i}B_i|> 10K \min(|A_i|,|B_i|).$$
Otherwise, we apply Theorem \ref{thm_shao} with parameters
$\varepsilon/s, 10K$ to the pair of sets $(A_i,B_i)$.  Thus, we produce a pair of sets $(A_{i+1},B_{i+1})$, satisfying
$A_{i+1} \subset A_i$,  $B_{i+1} \subset B_i$, \
$|A_{i+1}| \geq (1-\varepsilon/s)|A_i|$, \  $|B_{i+1}|\geq (1-\varepsilon/s)|B_i|$
and $|A_{i+1}+B_{i+1}| \leq |A_i+_{\Gamma_i}B_i|+\frac{\varepsilon}{s} \min(|A_i|,|B_i|)$.
We shall analyse separately the cases in which the process continues until the end and in
which the process stops before that.

\vspace{7pt}
{\bf Claim A.}
{\em
Suppose the process
stops early, say at step $j<s$. Then the pair of sets $(A_j,B_j)$ has the desired properties.
}

\begin{proof}
\vspace{7pt}
\noindent
\textbf{Case 1:} Consider first the
case when $|C^+_j|=|A_j+_{\Gamma_j}B_j|> 10K \min(|A_j|,|B_j|).$ For $x\in C_j^+$,
by construction we have that $r_j(x)=|(x-A_j)\cap B_j|\geq \alpha
|B_j|.$ If we choose elements $b_i$ uniformly at random
from $B_j$, for $1\leq i \leq c$, then
\begin{eqnarray*}
    \E_{b_1, \hdots, b_c \in B}|A_j+\{b_1, \hdots, b_c\})|
    &\geq& \sum_x \Prob_{b_1, \hdots, b_c \in B}(x \in  A_j+\{b_1, \hdots, b_c\})
      \geq \sum_{x \in C^+_j} 1-[1-\Prob_{b \in B}(x \in A_j+b)]^c\\
    &=& \sum_{x \in C^+_j} 1-[1- r_j(x)/|B_j|]^c
      \geq \sum_{x \in C^+_j} 1-(1-\alpha)^c\\
    &\geq& (1-(1-\alpha)^c) |C^+_j|
      \geq (1-(1-\alpha)^c) 10K \min(|A_j|,|B_j|) \\
    &\geq& 5K \min(|A_j|,|B_j|)
      \geq K \min(|A|,|B|).
\end{eqnarray*}

\vspace{7pt}
\noindent
\textbf{Case 2:} Consider now the case when $|\Gamma_j|<(1-\delta)|A_j||B_j|.$ By construction,
$\sum_{x\in C_j^-}r_j(x) \geq \delta|A_j||B_j|.$ Moreover, for
$x\in C_j^-$ we have $r_j(x)=|(x-A_j)\cap B_j|\leq \alpha
|B_j|.$
In particular, we deduce that
\[
1-[1- r_j(x)/|B_j|]^c \geq 1-[1- r_j(x)/|B_j|]^{\frac{10K}{\delta}} \geq \frac{5K}{\delta}r_j(x)/|B_j|.
\]
Again, if we choose elements $b_i$ uniformly at
random from $B_j$, for $1\leq i \leq c$, then
\begin{align*}
    \E_{b_1, \hdots, b_c \in B}|A_j+\{b_1, \hdots, b_c\}|
    & \geq \sum_x \Prob_{b_1, \hdots, b_c \in B}(x \in A_j+\{b_1, \hdots, b_c\})
      = \sum_{x \in C^-_j} 1-[1-\Prob_{b \in B}(x \in A_j+b)]^c\\
    & = \sum_{x \in C^-_j} 1-[1- r_j(x)/|B_j|]^c
      \geq \sum_{x \in C^-_j} \frac{5K}{\delta}r_j(x)/|B_j|\\
    & \geq \frac{5K}{\delta}\delta |A_j||B_j|/|B_j|
      \geq K|A|.
\end{align*}
We conclude the pair of sets $(|A_j|,|B_j|)$ has the desired properties, so Claim A is proved.
\end{proof}

We now turn to the case when the process does not stop early.

\vspace{7pt}
{\bf Claim B.} {\em
Suppose that the process continues until the terminal step $s$.
Then there is an index $j\leq s$ such that the  pair of sets $(A_j,B_j)$ has the desired
properties.}
\begin{proof}
Note that $|A_{1}+B_{1}| \leq |A_0+_{\Gamma_0}B_0|+\frac{\varepsilon}{s} \min(|A_0|,|B_0|) \leq 11K \min(|A_0|,|B_0|).$
Moreover, $11K \min(|A|,|B|)= 11K \min(|A_0|,|B_0|)\geq |A_{1}+B_{1}| \geq \hdots \geq |A_{s}+B_{s}| \geq 0.$ Therefore
$|A_{j+1}+B_{j+1}|\geq |A_{j}+B_{j}|-\frac{11K}{s-1}\min(|A|,|B|)$ for some index $j$,
$1 \leq j \leq s$.  We shall
show that the pair  $(A_j,B_j)$ has the desired properties.
Indeed, by construction,
\[
|A_{j}+_{\Gamma_j}B_{j}|+\frac{\varepsilon}{s}
\min(|A|,|B|) \geq |A_{j}+_{\Gamma_j}B_{j}|+\frac{\varepsilon}{s}
\min(|A_j|,|B_j|) \geq |A_{j+1}+B_{j+1}|.
\]
It follows that
\[
|A_j+_{\Gamma_j}B_j|\geq |A_{j}+B_{j}|- \big(\frac{11K}{s-1}+\frac{\varepsilon}{s} \big)\min(|A|,|B|) \geq \big(1- \frac{20K}{s} \big)|A_{j}+B_{j}|.
\]
If we choose elements $b_i$ uniformly at random from $B_j$, for $1\leq i \leq c$ then
\begin{equation*}
\begin{split}
    \E_{b_1, \hdots, b_c \in B}|A_j+\{b_1, \hdots, b_c\}|
    & \geq \sum_x \Prob_{b_1, \hdots, b_c \in B}(x \in A_j+\{b_1, \hdots, b_c\})
      \geq \sum_{x \in C^+_j} 1-[1-\Prob_{b \in B}(x \in A_j+b)]^c\\
    & = \sum_{x \in C^+_j} 1-[1- r_j(x)/|B_j|]^c
      \geq \sum_{x \in C^+_j} 1-(1-\alpha)^c\\
    & = (1-(1-\alpha)^c) |C^+_j|
      =  (1-(1-\alpha)^c) |A_j+_{\Gamma_j}B_j|\\
    & \geq (1-(1-\alpha)^c)(1- \frac{20K}{s}) |A_{j}+B_{j}|
      \geq (1-\varepsilon)|A_{j}+B_{j}|.
\end{split}
\end{equation*}
Thus the pair $(|A_j|,|B_j|)$ has the desired properties, so Claim B is proved.
\end{proof}

\vspace{5pt}
This concludes the proof of Theorem~\ref{implication_shao1-prov}: whether the process stops early or does not, the  pair of sets $(A_j,B_j)$ has the desired
properties. \hfill{$\square$}

\vspace{5pt}

We now digress briefly  to point out that Roth's theorem is a direct consequence of Theorem~$8'$. Recall that Roth's theorem~\cite{Roth} states that
{\em for every $\alpha>0$ there is an integer $k_0$ such that if $k\ge k_0$ then every set $A\subset [k]$ of size at least $\alpha k$ contains $3$ terms in arithmetic progression.}

\vspace{5pt}
\noindent
{\em Proof of Roth's Theorem.}
Fix $\alpha>0$ (where we may assume that $\alpha$ is sufficiently small). Pick $c$ as given by
Theorem~$10'$ with parameters $\frac{\alpha}{4}, \frac{2}{\alpha}$, and set
$k_0=2c[\alpha(1-\alpha)]^{-1}$.

\vspace{5pt}
Suppose for a contradiction that there exists a set $A$ that is free of three-term arithmetic
progressions, of density $\alpha$ in $[1,k]$ for some $k \geq k_0$. We apply Theorem~$10'$
with parameters $\frac{\alpha}{4}, \frac{2}{\alpha}$ to construct sets $A_1, A_2 \subset A$ of size at least $(1-\frac{\alpha}{4})|A|$ with the property that if we select $a_1, \hdots, a_c$ uniformly at random
from $A_2$ then $$\E_{a_1, \hdots, a_c \in A_2}\bigg|A_1+\{a_1, \hdots, a_c\}\bigg| \geq \min\bigg(\big(1-\frac{\alpha}{4}\big)|A_1+A_2|,\frac{2}{\alpha} |A| \bigg).$$
On the one hand, note that, as the only solutions to the equation $x+y=2z$ in $A$ are $x=y=z \in A$, the left-hand side is bounded from above by $$|A_1+A_2|-|A_1\cap A_2|+c\geq \E_{a_1, \hdots, a_c \in A_2}\bigg|A_1+\{a_1, \hdots, a_c\}\bigg|.$$
On the other hand, as $|A| \geq \alpha k$ and $|A_1+A_2| \leq 2k$, the right-hand side is bounded from below by
$$(1-\frac{\alpha}{4})|A_1+A_2| \geq |A_1+A_2| - \frac{\alpha}{2}k \geq |A_1+A_2|-\frac{|A|}{2}$$
Combining the two inequalities above we obtain $$ 2c \geq 2|A_1\cap A_2| -|A|.$$
Finally, as $A_1, A_2 \subset A$ satisfy $|A|\geq \alpha k$ and $|A_1|,|A_2| \geq (1-\frac{\alpha}{4})|A|$, we conclude
$$2c \geq \alpha (1-\alpha)k $$
which gives the desired contradiction. \hfill{$\square$}

\vspace{10pt}
In later sections we shall make considerable use of Theorem~\ref{implication_shao1-prov}.

\section{Sums in ${\mathbb Z}_p$: the Case of an Interval}

Our aim in this section is to prove Theorem \ref{intro_four_translates_Zp}. Although we are `only' considering
the case when $B$ is an interval, nevertheless the proof is not nearly as simple as one would expect.
Actually, in this section we also treat the case of general $B$ when $B$ is much smaller than $A$,
Theorem \ref{intro_general_three_translates_Zp}, as this seems to go hand-in-hand with the case when $B$ is an
interval. When in the next section we turn to the `most interesting' case, of general $B$ of comparable size to $A$, the
proofs will have their genesis in the proofs in this section,  although several more ideas will still be needed.

\vspace{5pt}
We need two technical results. In proving the first one we will make use of the
following corollary of Theorem~\ref{3k-4_Z_p}.

\begin{corollary}\label{Freiman_Zp}
Given $\beta, \gamma >0$ there is an $\varepsilon>0$ such that the following holds.
  Let $A$ and $B$ be subsets of $\mathbb{Z}_p$. Suppose that
\begin{equation}\label{eq21.-2}
    2\leq \min(|A|,|B|)\text{ and }|A|+|B| \leq (1-\beta)p
\end{equation}
and
\begin{equation}\label{eq21.-1}
    |A+B| \leq |A|+|B|-1 + \varepsilon \min(|A|,|B|).
\end{equation}
Then there are arithmetic progressions $P$ and $Q$ with the same common difference, containing $A$ and $B$, respectively,
 such that $$|P \Delta A|\leq  \gamma \min(|A|,|B|)  \text{ and }  |Q \Delta B|\leq \gamma \min(|A|,|B|).$$
\end{corollary}
\begin{proof}
  Fix $2^{-10}> \beta >\gamma>0$. Let $\eta$ be as given in Theorem~\ref{3k-4_Z_p}, and let
  $\varepsilon=2^{-10} \eta \min(\beta, \gamma)$. Finally, write $r=\lfloor\varepsilon\min(|A|,|B|)\rfloor$.\\

\vspace{5pt}
Theorem~\ref{vosper}, together with \eqref{eq21.-2} and \eqref{eq21.-1}, implies that we may assume that
\begin{equation}\label{eq21.0}
    1 \leq  r \leq \varepsilon p.
\end{equation}
In particular, we have
\begin{equation}\label{eq21.1}
    r \leq \eta p-2,
\end{equation}
and also
\begin{equation}\label{eq21.15}
    r \leq 2^{-10}\min(|A|,|B|) \leq \min(|A|,|B|)-3.
\end{equation}
If we let
\begin{equation*}\label{eq21.2}
    C=-(A+B)^{c}
\end{equation*}
then by \eqref{eq21.-2} and \eqref{eq21.-1} we see that
\begin{equation*}\label{eq21.3}
    |C|=p-|A+B|\geq p-(1-\beta+\varepsilon)p=(\beta-\varepsilon)p \geq 2^{-1}\beta p,
\end{equation*}
and by \eqref{eq21.0} we obtain
\begin{equation}\label{eq21.4}
    r\leq 2^{-9}|C| \leq |C|-3.
\end{equation}
Theorem~\ref{3k-4_Z_p} and  \eqref{eq21.1}, \eqref{eq21.15} and \eqref{eq21.4} now imply that there exist arithmetic
progressions $P$ and $Q$ with the same common difference, containing $A$ and $B$, respectively, such that
\begin{equation*}\label{eq21.5}
    |A\Delta P|\leq r \text{ and } |B \Delta Q|\leq r.
\end{equation*}
So we obtain
$$|A\Delta P|\leq \gamma \min(|A|,|B|) \text{ and } |B\Delta Q| \leq \gamma \min(|A|,|B|).$$

\vspace{5pt}
This concludes the proof of Corollary~\ref{Freiman_Zp}.
\end{proof}

\vspace{5pt}

Here then is the first technical result that we shall need in the proof of Theorem \ref{intro_four_translates_Zp}.
The reader should note that here the set $A$ has {\it small} intersection with a relevant
interval, in contrast to the usual requirement.

\begin{theorem}\label{stab_2pts}
  For all $\beta, \gamma >0$ there exist $\varepsilon, \alpha>0$ such that the following holds. Let $A$ and $B$ be subsets of $\mathbb{Z}_p$. Suppose that
\begin{equation}\label{eq37.-2}
    2\leq |B| \leq \alpha |A| \text{ and } |A|+|B|\leq (1-\beta)p
\end{equation}
and
\begin{equation}\label{eq37.-1}
    \max_{B' \in B^{(2)}}|A+B'| \leq |A|+|B|-1+\varepsilon |B|.
\end{equation}
Then there exist arithmetic progressions $P$ and $Q$ with the same common difference, and sizes at least $2^{10}|B|$ and
at most $\lfloor(1+\gamma)|B|\rfloor$, respectively, such that
$|A \cap  P| \leq \gamma|B|$ and  $B \subset Q$.
\end{theorem}

\begin{proof}
Fix $2^{-10}>\beta>\gamma>0$. Choose $\varepsilon'$ to be the output of Corollary~\ref{Freiman_Zp} with input $\beta, \gamma$. Choose $\varepsilon=2^{-3}\varepsilon'$ and $\alpha=\min(2^{-5}\beta \varepsilon, 2^{-23}\beta \gamma )$.  \\

\vspace{5pt}
Define
$$A'=
\begin{cases}
A \text{ if } |A| < p/2 \\
A^c \text{ if } |A| \geq p/2.
\end{cases}
$$
By \eqref{eq37.-2} and \eqref{eq37.-1} we deduce
\begin{equation}\label{eq37.0}
    2\leq |B| \leq \frac{\alpha}{\beta} |A'| \leq |A'| \leq p/2.
\end{equation}
and 
\begin{equation}\label{eq37.0.0.5}
    \max_{B' \in B^{(2)}}|A'+B'| \leq |A'|+|B|-1+\varepsilon |B|.
\end{equation}
In particular, we have
\begin{equation*}\label{eq37.0.1}
    p-1\geq |B-B|-1 \geq 2.
\end{equation*}
Define
\begin{equation*}\label{eq37.0.2}
    m=\big\lfloor \frac{p}{|B-B|-1} \big\rfloor \geq 1
\end{equation*}
and note that
\begin{equation}\label{eq37.0.3}
    m+1 \geq \frac{2|A'|}{|B-B|-1}.
\end{equation}
By the Cauchy-Davenport theorem, for $1\leq n \leq m $
we have
\begin{equation}\label{eq37.0.5}
    |n(B-B)| \geq n(|B-B|-1) \text{ and } |(m+1)(B-B)|=p.
\end{equation}
Therefore,  there exists a partition
\begin{equation}\label{eq37.0.6}
    \mathbb{Z}_p=C_1 \sqcup \hdots \sqcup C_m \sqcup C_{m+1}
\end{equation}
such that for $1\leq n \leq m$ we have
\begin{equation}\label{eq37.0.7}
   |C_n| \geq |B-B|-1
\end{equation}
and for $1\leq n \leq m+1$ we have
\begin{equation}\label{eq37.0.9}
    C_n \subset n(B-B).
\end{equation}
By \eqref{eq37.0.0.5}, for every $c \in B-B$ there exists a set $A'_c$  of size
\begin{equation*}\label{eq37.1}
    |A'_c| \geq |A'|-((1+\varepsilon) |B|-1)
\end{equation*}
such that
\begin{equation*}\label{eq37.2}
    A'_c \subset A' \text{ and } A'_c+c\subset A'.
\end{equation*}
It follows that for every $n \in \mathbb{N}$ and every $c \in n (B-B)$ there is a set
$A'_{n,c} $ of size
\begin{equation}\label{eq37.3}
    |A'_{n,c}| \geq |A'|-n((1+\varepsilon) |B|-1)
\end{equation}
such that
\begin{equation}\label{eq37.4}
     A'_{n,c} \subset A' \text{ and } A'_{n,c}+c\subset A'.
\end{equation}
By \eqref{eq37.0.6}, \eqref{eq37.0.9} and \eqref{eq37.4} we have
\begin{equation*}\label{eq37.5}
     \bigsqcup_{n \in [m+1]} \bigsqcup_{c \in C_n} \bigsqcup_{a \in A'_{n,c}}\{(a,a+c)\} \subset A'\times A'.
\end{equation*}
In particular, we deduce
\begin{equation}\label{eq37.5.5}
     |A'|^2 \geq \sum_{n \in [m+1]}\sum_{c \in C_n} \sum_{a \in A'_{n,c}}1.
\end{equation}
On the one hand,  \eqref{eq37.0.6}, \eqref{eq37.3} and \eqref{eq37.5.5} imply
\begin{equation*}\label{eq37.5.7}
     |A'|^2 \geq p(|A'|-(m+1)((1+\varepsilon)|B|-1)).
\end{equation*}
and \eqref{eq37.0} further implies
\begin{equation}\label{eq37.5.8}
     m+1 \geq \frac{\beta}{4\alpha} .
\end{equation}
On the other hand, \eqref{eq37.0.7}, \eqref{eq37.3} and \eqref{eq37.5.5}  imply
\begin{eqnarray}\label{eq37.6}
         |A'|^2
         &\geq& \sum_{n \in [m]} (|B-B|-1)\max\bigg(0,|A'|-n((1+\varepsilon)|B|-1)\bigg)\\
         &=& \sum_{n \in [m]} ((1+\varepsilon)|B|-1)(|B-B|-1)\max\bigg(0,\frac{|A'|}{(1+\varepsilon)|B|-1}-n\bigg).
\end{eqnarray}
We distinguish two cases.

\vspace{7pt}
{\bf Case A: }{\em Suppose
\begin{equation}\label{eq37.7}
m+1 \leq \frac{|A'|}{(1+\varepsilon)|B|-1}.
\end{equation}
}
In this case, from inequalities \eqref{eq37.0.3}, \eqref{eq37.5.8} and \eqref{eq37.7} we deduce that
\begin{equation}\label{eq37.8}
\sum_{n \in [m]}\max\bigg(0,\frac{|A'|}{(1+\varepsilon)|B|-1}-n\bigg)
\ge \big(2-\frac{8\alpha}{\beta}\big) \frac{|A'|}{|B-B|-1}\bigg(\frac{|A'|}{(1+\varepsilon)|B|-1}-\frac{|A'|}{|B-B|-1}\bigg).
\end{equation}
Combining \eqref{eq37.6} and \eqref{eq37.8}, we obtain
\begin{equation}\label{eq37.9}
\begin{split}
1 \geq (2-\frac{8\alpha}{\beta}) (1- \frac{(1+\varepsilon)|B|-1}{|B-B|-1}).
\end{split}
\end{equation}

\vspace{7pt}
{\bf Case B: }{\em Suppose
\begin{equation}\label{eq37.11}
m+1 > \frac{|A'|}{(1+\varepsilon)|B|-1}.
\end{equation}
}
In this case,
by \eqref{eq37.0} and \eqref{eq37.11} we deduce
\begin{equation}\label{eq37.12}
    \begin{split}
        \sum_{n \in [m]}\max\bigg(0,\frac{|A'|}{(1+\varepsilon)|B|-1}-n\bigg) &=
        \sum_{n \in \mathbb{N}}\max\bigg(0,\frac{|A'|}{(1+\varepsilon)|B|-1}-n\bigg)\\
        &\geq 2^{-1}\frac{|A'|}{(1+\varepsilon)|B|-1}(\frac{|A'|}{(1+\varepsilon)|B|-1}-1)\\
        &\geq (2^{-1}-\frac{4\alpha}{\beta}) (\frac{|A'|}{(1+\varepsilon)|B|-1})^2
    \end{split}
\end{equation}
By combining \eqref{eq37.6} and \eqref{eq37.12}, after simplifying we obtain
\begin{equation}\label{eq37.13}
    1\geq (2^{-1}-\frac{4\alpha}{\beta}) \frac{|B-B|-1}{(1+\varepsilon)|B|-1}.
\end{equation}
In both cases, by \eqref{eq37.9} and \eqref{eq37.13}, after rearranging we get
\begin{equation}\label{eq37.14}
\begin{split}
|B-B|-1
&\leq (2+\frac{16 \alpha}{\beta})((1+\varepsilon)|B|-1)\\
&\leq (2+\frac{16\alpha}{\beta}+4\varepsilon)|B|-2.
\end{split}
\end{equation}
By Corollary~\ref{Freiman_Zp}, together with \eqref{eq37.0} and \eqref{eq37.14}, it follows that there is an arithmetic progression $Q$ containing $B$ such that $$ |Q \setminus B| \leq \gamma|B|.$$

We may assume that $Q$ is an interval. It is now easy to see that the interval $J=[-|B|+1,|B|-1]$ of size
$2|B|-1$ is contained inside $B-B$. Indeed, suppose instead that some $l$, with $0<l<|B|$, does not belong to
$B-B$. If $l>|B|/2$ then at most one of each pair $(t,t+l)$, as $t$ ranges from $\min Q$ to $\max Q - l$, can
belong to $B$, and this contradicts the fact that $|B| \geq |Q|/(1+\gamma)$. On the other hand, if
$l \leq |B|/2$ then we consider the fibres of $Q$ mod $l$: each fibre has size at least 2, and no fibre can have
two consecutive points in $B$. It follows that each fibre meets $B$ in at most $2/3$ of its points, and again this
contradicts the fact that $|B| \geq |Q|/(1+\gamma)$.

Hence we may also assume that $B$ is an interval and that $J=B-B$.

\noindent
By \eqref{eq37.-1} we have $\max_{c\in J} |(A+c)\setminus A| \leq |B|-1+\varepsilon|B|$, so
$\max_{c\in 2^{10}J} |(A+c)\setminus A| \leq 2^{10}(|B|-1+\varepsilon|B|)$.
Set
$K=\{x \in A^c \text{ : } |(x-2^{10}J) \cap A| \leq 2^{-11}\gamma |2^{10}J|    \}$,
and note that
\begin{eqnarray*}
    2^{11}|B|\geq 2^{10}(|B|-1+\varepsilon|B|)
    &\geq& \E_{c\in 2^{10}J}|(A+c)\setminus A|
    =\sum_{x \in A^c} \Prob_{c \in 2^{10}J} \bigg( x \in A+c \bigg)\\
    &=&\sum_{x \in A^c} \frac{|(x-2^{10}J)\cap A|}{|2^{10}J|}
    \geq 2^{-11}\gamma |A^c\setminus K|.\\
\end{eqnarray*}
By \eqref{eq37.-2}, we have
\[
    |K| \geq |A^c|-2^{22}\gamma^{-1}|B|
    \geq \beta p -  2^{22} \gamma^{-1} \alpha p
    >0.
\]
Therefore $K\not= \emptyset$, so we can pick an element $x \in K$. By construction,
$P=x-2^{10}J$
is an interval of size
$|P| = 2^{10}(2|B|-1)-2^{10}+1 \in (2^{10}|B|, 2^{11}|B|)$
such that
$|A\setminus P^c|=|A\cap P| \leq \gamma |B|.$
This concludes the proof of Theorem~\ref{stab_2pts}.
\end{proof}

\vspace{5pt}
Here is the second technical result that we shall need.
\begin{theorem}\label{general_3pts}
Let $A$ and $B$ be subsets of $\mathbb{Z}_p$ and let $I=[p_l, p_r]$ and $J=[q_l, q_r]$ be intervals of $\mathbb{Z}_p$ satisfying
\begin{equation}\label{eq39.-2}
2\leq |B| \leq \min(|A|, p/2^{11}) \text{, } |I|=2^{10}|B| \text{ and } |J| \leq (1+2^{-10})|B|
\end{equation}
and
\begin{equation}\label{eq39.-1}
|A\setminus I^c| \leq 2^{-10} |B| \text{ and } \{q_l,q_r\} \subset B \subset J.
\end{equation}
Then there exists
$B'\in B^{(3)}$ such that $A+B'|\geq |A|+|B|-1$.
\end{theorem}

\begin{proof}
By \eqref{eq39.-2}, there is a partition $I=I_{-2}\sqcup I_{-1}\sqcup I_0\sqcup I_1\sqcup I_2$ of interval $I$ into five consecutive disjoint intervals such that
$p_l \in I_{-2} \text{ and } p_r \in I_2$, \
$ |I_0|=|I_{-2}|= |I_{2}| = |J|$,
and
$ |I_0\cap A| \leq 2^{-4} |(I_{-1}\sqcup I_0 \sqcup I_1)\cap A|.$
Let $I_{\infty}$ be the complement of $I$: \ $I_{\infty}=I^c.$
For a subset $S\subset \{-2,-1,0,1,2, \infty\}$ write
$A_S=\sqcup_{x\in S}(A\cap I_x).$
By construction and by \eqref{eq39.-1}, for $x, y \in B$, we have the following three assertions:
\begin{equation}\label{eq39.1}
    \text{the sets} \ \  A_{-1}+q_r,\ A_1+q_l \ \ \text{and} \ \ A_{-2,\infty, 2}+y \ \ \text{are  disjoint},
\end{equation}

\begin{equation}\label{eq39.2}
    \text{the sets} \ \ A_{-1,0,1}+x \ \ \text{and} \ \  A_{\infty}+y \ \ \text{are  disjoint},
\end{equation}
and
\begin{equation}\label{eq39.25}
    |A_0| \leq 2^{-4}|A_{-1,0,1}|.
\end{equation}

\vspace{5pt}
We will need two claims.

\vspace{7pt}
{\bf Claim A. }{\em
If $|A_{-2,\infty,2}|\geq |B|-1$ then
$$\E_{b\in B \setminus \{q_r\}}\bigg|A_{-2,\infty,2}+\{q_l,q_r, b\} \bigg| \ge |A_{-2,\infty,2}|+|B|-1.$$
If $|A_{-2,\infty,2}|\leq |B|-1$ then either
$$\E_{b\in B \setminus \{q_r\}}\bigg|A_{-2,\infty,2}+\{q_l,q_r, b\} \bigg|  \ge 2|A_{-2,\infty,2}| + 2^{-3}(|B|-1-|A_{-2,\infty,2}|)$$
or
$$ \E_{b_2,b_3\in B \setminus \{q_r\}}\bigg|A_{-2,\infty,2}+\{q_l,b_2, b_3\} \bigg|  \ge 2.5|A_{-2,\infty,2}|.$$
}
\begin{proof}
Claim A refers to subsets $A_{-2,\infty,2}\subset I_{-2}\sqcup I_{\infty}\sqcup I_2$ and $B\subset J$. Given that the intervals $I_{-2}\sqcup I_{\infty}\sqcup I_2$ and $J$ satisfy $|I_{-2}\sqcup I_{\infty}\sqcup I_2|+|J|<p$, we may assume without loss of generality that the ambient space is $\mathbb{Z}$ rather than $\mathbb{Z}_p$.\\

\vspace{5pt}
First, suppose that $|A_{-2,\infty,2}|\geq |B|-1$. By Theorem~\ref{technical_three_translates_Z},  \eqref{eq39.-2} and \eqref{eq39.-1} we have
$$\E_{b\in B \setminus \{q_r\}}\bigg|A_{-2,\infty,2}+\{q_l,q_r, b\} \bigg|  \ge |A_{-2,\infty,2}|+|B|-1.$$
Second, suppose that $|A_{-2,\infty,2}|\leq |B|-1$ and $8|\pi_{q_r-q_l}(A_{-2,\infty,2})| \geq  |\pi_{q_r-q_l}(B)|$. Again, from Theorem~\ref{technical_three_translates_Z},  \eqref{eq39.-2} and \eqref{eq39.-1} we see that
\begin{eqnarray*}
        \E_{b\in B \setminus \{q_r\}}\bigg|A_{-2,\infty,2}+\{q_l,q_r, b\} \bigg| &\ge& 2|A_{-2,\infty,2}|+|\pi_{q_r-q_l}(A_{-2,\infty,2})|(1-\frac{|A_{-2,\infty,2}|}{|B|-1})\\
        &=&2|A_{-2,\infty,2}|+(|B|-1-|A_{-2,\infty,2}|)\frac{|\pi_{q_r-q_l}(A)|}{|\pi_{q_r-q_l}(B)|}\\
        &\geq& 2|A_{-2,\infty,2}|+2^{-3}(|B|-1-|A_{-2,\infty,2}|).
\end{eqnarray*}
Third, suppose that $|A_{-2,\infty,2}|\leq |B|-1$ and $8|\pi_{q_r-q_l}(A_{-2,\infty,2})| \leq  |\pi_{q_r-q_l}(B)|$. Then Lemma~\ref{residue.vs2},  \eqref{eq39.-2} and \eqref{eq39.-1} imply that
\begin{equation*}
       \E_{b_2,b_3\in B \setminus \{q_r\}}\bigg|A_{-2,\infty,2}+\{q_l,b_2, b_3\} \bigg|  \ge 2.5|A_{-2,\infty,2}|,
\end{equation*}
proving Claim A.
\end{proof}

\vspace{7pt}
{\bf Claim B. }{\em
We have
$$  \E_{b\in B \setminus \{q_r\}}\bigg|(A_{-1,0,1}+\{q_l,q_r, b\}) \setminus ( A_{-2,\infty,2}+\{q_l,q_r, b\} ) \bigg| \geq (2-2^{-3})|A_{-1,0,1}|$$
and
$$  \E_{b_2,b_3\in B \setminus \{q_r\}}\bigg|(A_{-1,0,1}+\{q_l,b_2, b_3\}) \setminus ( A_{-2,\infty,2}+\{q_l,b_2, b_3\} ) \bigg| \geq (2-2^{-3})|A_{-1,0,1}|.$$
}
\begin{proof}
By \eqref{eq39.1} and \eqref{eq39.2}, we have
\[
\E_{b\in B \setminus \{q_r\}}\bigg|(A_{-1,0,1}+\{q_l,q_r, b\}) \setminus ( A_{-2,\infty,2}+\{q_l,q_r, b\} ) \bigg|
\]
 \[
\geq \bigg|A_{-1}+q_r\bigg|+ \bigg|A_{1}+q_l\bigg|+ \E_{b\in B \setminus \{q_r\}}\bigg|(A_{-1,0,1}+b) \setminus  (A_{-2,-1,1,2}+\{q_l,q_r\})\bigg|.
\]
Also, \eqref{eq39.-2} and \eqref{eq39.-1} imply that
\[
        \E_{b\in B \setminus \{q_r\}}\bigg|(A_{-1,0,1}+b) \setminus (A_{-2,-1,1,2}+\{q_l,q_r\})\bigg| \geq \sum_{a\in A_{-1,0,1}} \Prob_{b\in B \setminus \{q_r\}} \bigg(a+b \not\in A_{-2,-1,1,2}+\{q_l,q_r\} \bigg)
\]
\[
\geq |A_{-1,0,1}| \frac{|B|-1-2|A_{-2,-1,1,2}|}{|B|-1}
        \geq (1-2^{-4})|A_{-1,0,1}|.
\]
Combining these two inequalities, we get
\[
\E_{b\in B \setminus \{q_r\}}\bigg|(A_{-1,0,1}+\{q_l,q_r, b\}) \setminus ( A_{-2,\infty,2}+\{q_l,q_r, b\} ) \bigg| \geq |A_{-1}|+|A_1|+(1-2^{-4})|A_{-1,0,1}|.
\]
From \eqref{eq39.25} we conclude that
\[
  \E_{b\in B \setminus \{q_r\}}\bigg|(A_{-1,0,1}+\{q_l,q_r, b\}) \setminus ( A_{-2,\infty,2}+\{q_l,q_r, b\} ) \bigg| \geq (2-2^{-3})|A_{-1,0,1}|.
\]

\vspace{5pt}
By \eqref{eq39.1} and \eqref{eq39.2}, this gives
\[
\E_{b_2,b_3\in B \setminus \{q_r\}}\bigg|(A_{-1,0,1}+\{q_l,b_2, b_3\}) \setminus ( A_{-2,\infty,2}+\{q_l,b_2, b_3\} ) \bigg|\ge
\]
\[
\E_{b_2, b_3\in B \setminus \{q_r\}}\bigg|(A_{-1,0,1}+b_3) \setminus  (A_{-2,-1,0,1,2}+\{q_l, b_2\})\bigg|+ \E_{b_2, b_3\in B \setminus \{q_r\}}\bigg|(A_{-1,0,1}+b_2) \setminus  (A_{-2,-1,0,1,2}+\{q_l, b_3\})\bigg|.
\]
By \eqref{eq39.-2} and \eqref{eq39.-1}, for fixed $b_3\in B$ we have
\[
        \E_{b_2\in B \setminus \{q_r\}}\bigg|(A_{-1,0,1}+b_2) \setminus  (A_{-2,-1,0,1,2}+\{q_l,b_3\})\bigg| \geq \sum_{a\in A_{-1,0,1}} \Prob_{b_2\in B \setminus \{q_r\}} \bigg(a+b_2 \not\in A_{-2,-1,0,1,2}+\{q_l,b_3\} \bigg)
\]
\[
        \geq |A_{-1,0,1}| \frac{|B|-1-2|A_{-2,-1,0,1,2}|}{|B|-1}
        \geq (1-2^{-4})|A_{-1,0,1}|.
\]
Combining the last two inequalities, we conclude that
\[
\E_{b_2,b_3\in B \setminus \{q_r\}}\bigg|(A_{-1,0,1}+\{q_l,b_2, b_3\}) \setminus ( A_{-2,\infty,2}+\{q_l,b_2, b_3\} ) \bigg| \geq (2-2^{-3})|A_{-1,0,1}|.
\]

\vspace{5pt}
This proves Claim B.
\end{proof}

\vspace{7pt}
Returning to the proof of Theorem~\ref{general_3pts},
if $|A_{-2,\infty,2}|\geq |B|-1$, then from Claims A and B we have
$$\E_{b \in B \setminus \{q_r\}} \bigg| A+\{q_l, q_r,b\} \bigg| \geq $$ $$\E_{b\in B \setminus \{q_r\}}\bigg|A_{-2,\infty,2}+\{q_l,q_r, b\} \bigg|+ \E_{b\in B \setminus \{q_r\}}\bigg|(A_{-1,0,1}+\{q_l,q_r, b\}) \setminus ( A_{-2,\infty,2}+\{q_l,q_r, b\} ) \bigg| \geq$$
$$  |A_{-2,\infty,2}|+|B|-1+(2-2^{-3})|A_{-1,0,1}|\geq $$
$$|A|+|B|-1.$$
If $|A_{-2,\infty,2}|< |B| \leq |A|$, then from Claims A and B we conclude that either
$$\E_{b \in B \setminus \{q_r\}} \bigg| A+\{q_l, q_r,b\} \bigg| \geq $$ $$\E_{b\in B \setminus \{q_r\}}\bigg|A_{-2,\infty,2}+\{q_l,q_r, b\} \bigg|+ \E_{b\in B \setminus \{q_r\}}\bigg|(A_{-1,0,1}+\{q_l,q_r, b\}) \setminus ( A_{-2,\infty,2}+\{q_l,q_r, b\} ) \bigg| \geq$$
$$2|A_{-2,\infty,2}| +2^{-3}(|B|-1-|A_{-2,\infty,2}|)+(2-2^{-3})|A_{-1,0,1}| = $$
$$(2-2^{-3})|A|+2^{-3}(|B|-1) \geq$$
$$|A|+|B|-1,$$
or
$$\E_{b_2,b_3 \in B \setminus \{q_r\}} \bigg| A+\{q_l, b_2,b_3\} \bigg| \geq $$
$$\E_{b_2,b_3\in B \setminus \{q_r\}}\bigg|A_{-2,\infty,2}+\{q_l,b_2, b_3\} \bigg|+ \E_{b_2,b_3\in B \setminus \{q_r\}}\bigg|(A_{-1,0,1}+\{q_l,b_2, b_3\}) \setminus ( A_{-2,\infty,2}+\{q_l,b_2, b_3\} ) \bigg| \geq$$
$$2.5|A_{-2,\infty,2}|+(2-2^{-3})|A_{-1,0,1}| \geq $$
$$|A|+|B|-1.$$
The last inequality follows from the fact that $|A_{1,0,1}|  \leq |A\setminus I^c| \leq  2^{-1}|A|$.

\vspace{5pt}
This finishes the proof of Theorem~\ref{general_3pts}.
\end{proof}

\vspace{5pt}
We are now ready to prove Theorems \ref{intro_general_three_translates_Zp}
and  \ref{intro_four_translates_Zp}.

\vspace{7pt}
\noindent{\bf Theorem~\ref{intro_general_three_translates_Zp}.} {\em
For every $\beta>0$ there exists $\alpha>0$ such that the following holds.
Whenever $A$ and $B$ are non-empty subsets of
$\mathbb{Z}_p$ with $|B| \leq \alpha |A|$ and $|A|+|B| \leq (1-\beta)p$,
there exist
elements $b_1,b_2,b_3 \in B$ such that
$$|A+\{b_1, b_2,b_3\}| \ge |A|+|B|-1.$$
}

\begin{proof}
  Fix $2^{-10}>\beta>0$. Choose $\gamma=2^{-10}$. Let $\varepsilon', \alpha'$ be the output of Theorem~\ref{stab_2pts} with input $\beta, \gamma$. Choose $\alpha=\min(2^{-11},\alpha')$. Assume for a
  contradiction that
\begin{equation}\label{eq40.0}
    \max_{B'\in B^{(3)}} |A+B'|<|A|+|B|-1.
\end{equation}
Theorem~\ref{stab_2pts} and the conditions on $|A|$ and $|B|$ imply that there are arithmetic progressions $I$ and $J$ with the same common difference and sizes at least $2^{10}|B|$ and at most $\lfloor(1+\gamma)|B|\rfloor$, respectively, such that $|A\setminus I^c| \leq 2^{-10} |B| \text{ and } |B\setminus J| =0.$
We may assume without loss of generality that $I=[p_l, \hdots, p_r]$ and $J=[q_l, \hdots, q_r]$ are actually intervals
satisfying
$|I|=2^{10}|B|$,  $|J| \leq (1+2^{-10})|B|$,
$|A\setminus I^c| \leq 2^{-10} |B|$ and $\{q_l,q_r\} \subset B \subset J$.
Theorem~\ref{general_3pts} and these conditions on $|A|$ and $|B|$ imply that
$ \max_{B'\in B^{(3)}} |A+B'| \geq |A|+|B|-1.$
\end{proof}

\vspace{7pt}
\noindent {\bf Theorem~\ref{intro_four_translates_Zp}.} {\em
Let $A$ and $B$ be non-empty subsets of $\mathbb{Z}_p$ with
$ |B| \leq |A| \leq 2^{-20}p$ and $B$ an interval.
Then $B$ has three elements, $b_1, b_2$ and $b_3$,
such that $|A+\{b_1, b_2,b_3\}| \ge |A|+|B|-1.$
}

\begin{proof}
Consider the interval $J=B$. By hypothesis, there exists an interval $I$ of size $2^{10}|B|$ such that $|A\setminus I^c| \leq 2^{-10}|B|$ (consider a random interval $I$). By Theorem~\ref{general_3pts} we conclude that
$$ \max_{B'\in B^{(3)}} |A+B'| \geq |A|+|B|-1.$$
\end{proof}

\vspace{5pt}
In the next section we turn our attention to the general case.

\section{Sums in ${\mathbb Z}_p$: the General Case}

We now focus on our main aim, Theorem~\ref{stability_in_Zp_for_intro}.
Actually, in the form (below) that we state Theorem~\ref{stability_in_Zp_for_intro} we will allow $r$ to be negative, and since in that case the
conclusion of Theorem~\ref{stability_in_Zp_for_intro} is trivially false Theorem \ref{translates_of_A_Zp} follows immediately. This avoids us writing down
what is essentially the same long argument twice. Of course, Theorem \ref{translates_of_A_Zp} contains
Theorem \ref{translates_of_A_Zp_equality} and Theorem \ref{intro_few_translates_Z}
as special cases.

\vspace{5pt}
What we actually prove is the following strengthening of Theorem~\ref{stability_in_Zp_for_intro}.

\begin{theorem}\label{strong_stab_B_Z_p}
  For every $\beta>0$ there exists $\varepsilon>0$ such that for every $\alpha>0$ there is a value of $c$ for which the
  following holds. Let $A$ and $B$ be subsets of $\mathbb{Z}_p$. Suppose that
\begin{equation}\label{eq34.-2}
    2\leq \min(|A|,|B|) \text{, } \alpha |B| \leq |A| \leq \alpha^{-1} |B| \text{ and } |A|+|B|\leq (1-\beta)p
\end{equation}
and
\begin{equation}\label{eq34.-1}
    \max_{B'\in B^{(c)}}|A+B'|= |A|+|B|-1+r \leq |A|+|B|-1+\varepsilon \min(|A|,|B|)
  \end{equation}
for some integer $r$.
Then $B$ is contained in an arithmetic progression of size $|B|+r$.
\end{theorem}

\vspace{5pt}

As remarked earlier, we cannot insist that $A$ is also contained in a short arithmetic
progression. However,
using the same method as above, one can show that there is an arithmetic progression of size
$|A|+r$ such that $A$
is contained in this progression except possibly for at
most $r_0\le \delta(c)r$ terms, where $\delta(c) \to 0$
as $c\to \infty$. This result is optimal in the sense
that $r_0/r \to 0$ need not hold if $c \not\to \infty$.

\vspace{10pt}
As the proof of Theorem~\ref{strong_stab_B_Z_p} is rather involved, we give an informal overview before
proceeding to the proof itself. Some of the statements in this overview are somewhat imprecise,
but they will give the rough picture.

Let us assume that $\varepsilon>0$ is sufficiently small and the sets
that $A$ and $B$ have sizes that are say within a factor of 10 of $n$, where $n$ is much smaller
than $p$.

Let $Q$ be a shortest arithmetic progression containing $B$. Our aim is to show that there is a subset $B'$ of
$B$ of size $c$ such that $|A+B'|$ is bounded from below either by $|A|+|Q|-1$ or
by $|A|+|B|-1+\varepsilon n$.
The two cases depend upon whether or not both $A$ and $B$ are close to arithmetic progressions.

The arguments for $A$ and $B$ turn out to be very similar, so we consider first the case when
say $B$ is
far from every arithmetic progression $R$ (using the symmetric difference as our measure),
meaning say that $|B \Delta R|> 100 \varepsilon n$. Then,
by applying Theorem~\ref{implication_shao1-prov},
we have subsets $A^*$ and $B^*$ of $A$ and $B$, of sizes at
least $(1 - \varepsilon)|A|$ and $(1 - \varepsilon)|B|$ respectively, such that
$|A+B'| > (1-\varepsilon) |A^*+B^*|$
for some subset $B'$ of $B$ of size $c$. Note that the property of being far
from any arithmetic progression is robust under small perturbations of the set.
Thus, for any arithmetic progression $R$, we have
$|B^* \Delta R|>100 \varepsilon n - \varepsilon |B| > 50 \varepsilon n$.
By Freiman's theorem we have that
$|A^*+B^*| \geq |A^*|+|B^*|+50 \varepsilon n > |A|+|B|+30 \varepsilon n$
which implies that   $|A+B'|> |A|+|B|+\varepsilon n$, as claimed.

Thus we have reduced our problem to the case when both $A$ and $B$ are close (in
symmetric difference) to some progressions $P$ and $R$, i.e.
$|A\Delta P|$ and $|B\Delta R|$ are both at most $100 \varepsilon n$.
A careful examination of the above argument actually shows
that $P$ and $R$ have the same common difference, and so from now on we
may assume that $P$ and $R$ are intervals.

We now want to argue that $B$ is not only a subset of the arithmetic progression $Q$,
the arithmetic progression chosen at the beginning  of the argument,
but is also close to it (again in symmetric difference), say $|Q \setminus
B| < 10000 \varepsilon n$. If this is not the case, then $B$ cannot be contained
in the interval obtained from $R$ by blowing it up by $5000 \varepsilon n$,
i.e. there is a point $b\in B$ at distance at least $5000 \varepsilon n$
from $R$.

Let $A_1$ be the part of $A$ inside $P$ and $B_1$ be the part of $B$ inside $R$.
The sets $A_1$ and $B_1$ are about the same size as $|A|$ and $|B|$: $|A_1|$ and $|B_1|$ are at
least $(1 - \varepsilon)|A|$ and $(1 - \varepsilon)|B|$ respectively. 
Applying Theorem~\ref{implication_shao1-prov} again, we obtain subsets $A^*$ and $B^*$ of $A_1$ and
$B_1$, of sizes at least $(1 - \varepsilon)|A_1|$ and $(1 - \varepsilon)|B_1|$ respectively, such that
$| A_1+B_1'| > (1-\varepsilon) |A^*+B^*|$ for some subset $B_1'$ of $B_1$ of size $c$. It
follows that $|A_1+B_1'|> |A|+|B|- 1000 \varepsilon n$ and moreover that $A_1+B_1$ is a
subset of $P+R$. Now using the additional point $b$ which is far from $R$ we
get that $P+b$ sticks out from $P+R$ by many points, at least $5000 \varepsilon n$, and hence
$A_1+b$ also sticks out from $P+R$ by many points, namely at least $4000 \varepsilon n$. 
Therefore
$| A_1+ (B_1' \cup b)| > |A|+|B|+3000 \varepsilon n$, say.

Hence both $|A \Delta P|$ and $|Q \setminus B|$ are small, namely at most
$10000 \varepsilon n$. It is now that the most significant obstacle comes into play:
it could be that $|Q \setminus B|$ is small, like say $\sqrt n$.

For simplicity assume that $A$ is contained in $P$, and recall that $Q$ is the shortest
arithmetic progression containing $B$. In this situation, in the sum $A+B$ there is no wraparound, since
$|P|+|Q|<(1+10^5\varepsilon)(|A|+|B|)<p$. Hence we can view $A$ and $B$ as being
subsets of ${\mathbb Z}$ rather than ${\mathbb Z}_p$.

When $|A| \ge |B|$, we take a random subset $B'$ of $B$ formed from the leftmost and
rightmost points of $B$ and an additional random point of $B$. Recalling Theorem 1,
we find that ${\mathbb E}|A+B'| \geq |A|+|B|-1$.

We adopt an analogous strategy in general: let say
$|B|/100 \leq |A| \leq |B|$. 
Consider the following ideal
case. Suppose that there is an arithmetic progression $M$ with step size $d$, containing the endpoints
of $B$, such that
$d$ is much smaller than $|A|$ and with $B
\cap M = Q \cap M$.  Set $B_x = B \cap (x \mod d)$ and assume $B_0= B
\cap M$.    In this case, an analysis similar to that in the proof of Theorem 1 shows that, for
a random $x \in {\mathbb Z}_d$, if we set $B'=B_0 \cup B_x$ then
${\mathbb E}|A+B'| \ge |A|+|Q|-1$.
For this argument to work, it is important  that $A$ is close to $P$ and $B$ is
close to $Q$. In general there need not be an arithmetic progression $M$ with step size $d$ (containing
the endpoints of $B$) such that
$d$ is much smaller than $|A|$ and $B \cap M = Q \cap M$. However, as it turns out, we can
always
divide $A$ and $B$ into three consecutive blocks such that in each block we can
perform this construction. For this argument to work we do also need that $B$ is
close to $Q$.

\vspace{5pt}
We now embark on the actual proof of Theorem~\ref{strong_stab_B_Z_p}. First we
need a couple of straightforward estimates.

\begin{lemma}\label{change_c}
  For any $\mu>0$ there exists $c_0$ such that for all $c\geq c_0$ the following holds. Let $A,B$ be finite subsets of an abelian group. If $C$ is a subset of $B$ of size $|C|\geq 2|B|/3$ then
  $$\ \ \ \ \ \ \ \E_{B' \in B^{(c)}}|A+(B'\cap C)|\geq (1-\mu)\E_{C'\in C^{(\frac{c}{2})}}|A+C'|.
  \ \ \ \ \ \ \square$$
\end{lemma}
\begin{lemma}\label{change_c_again}
For all $m\in \mathbb{N}$ and $\mu>0$ there exists $c_0$ such that for all $c\geq c_0$ the following holds. Let $A$, $B$ and $C$ be finite subsets of an abelian group. If $B_1, B_2, \hdots B_m$ are subsets of $B$ of size $|B_i|\geq \mu |B|$ then
$$\ \ \ \ \ \ \E_{B' \in B^{(c)}}\bigg|(A+B')
\setminus C\bigg| \geq \frac{3}{4}\E_{\substack{ b^1_{j}, b_j^2 \in B_{j} \\ j \in [m] }}\bigg|(A+ \cup_{j \in [m]} \{b_{j}^1, b_{j}^2\})\setminus C\bigg|. \ \ \ \ \ \ \square  $$
\end{lemma}

\vspace{5pt}

We also need the following weak stability result.

\begin{theorem}\label{weak_stab_Zp}
  For all $\beta, \gamma >0$ there exists $\varepsilon>0$ such that for all $\alpha>0$ there is a
  value of $c$ for which the following holds. Let $A$ and $B$ be subsets of $\mathbb{Z}_p$
  with
\begin{equation}\label{eq22.-2}
    2\leq \min(|A|,|B|) \text{ , } \alpha |B| \leq |A| \leq \frac{1}{\alpha}|B| \text{ and } |A|+|B| \leq (1-\beta)p.
\end{equation}
Then there exists a element $b \in B$ such that if
\begin{equation}\label{eq22.-1}
    \E_{B'\in B^{(c)}}\bigg|A+(B'\cup \{b\})\bigg| \leq |A|+|B|-1 + \varepsilon \min(|A|,|B|)
\end{equation}
then there exist arithmetic progressions $P$ and $Q$ with the same common difference such that  $$|A\Delta P| \leq \gamma \min(|A|,|B|) \text{ , } |B\Delta Q| \leq \gamma \min(|A|,|B|) \text{ and } B \subset Q.$$
\end{theorem}

\begin{proof}[Proof of Theorem \ref{weak_stab_Zp}]
Fix $2^{-10}>\beta>\gamma>\alpha>0$. Let $\varepsilon_1$ be the output of Theorem~\ref{Freiman_Zp} with input $\beta,2^{-11}\gamma$. Let $\varepsilon=\min(2^{-2}\varepsilon_1,2^{-6}\gamma)$. Let $\mu=2^{-12}\min(\alpha \varepsilon, \alpha \gamma)$. Let $c_0$ be the output of Lemma~\ref{change_c} with input $\mu$. Let $c_1$ be the output of Theorem~\ref{implication_shao1-prov} with input $(4\alpha^{-1},\mu)$. Let $c=2\max(c_0,c_1)$. This time we partition our proof into three Claims.

\vspace{5pt}
{\bf Claim A. }{\em
There exist arithmetic progressions $U,V$ with the same common difference and
\begin{equation}\label{eq22.05}
    |A\Delta U| \leq \frac{\gamma}{2^{10}} \min(|A|,|B|) \text{ and } |B\Delta V| \leq \frac{\gamma}{2^{10}} \min(|A|,|B|).
\end{equation}
}

\noindent
{\em Proof of Claim A.}
By Theorem~\ref{implication_shao1-prov}, there are subsets $A^*\subset A$ and $B^*\subset B$ with
\begin{equation}\label{eq22.1}
    |A^*|\geq \lceil (1-\mu)|A|\rceil  \text{ and } |B^*|\geq \lceil (1-\mu)|B|\rceil
\end{equation}
such that
\begin{equation}\label{eq22.2}
    \E_{B'\in B^{(c)}}| A+B'| \geq \min \bigg((1-\mu)|A^*+B^*|\text{, } \frac{4}{\alpha}|A|\text{, } \frac{4}{\alpha}|B|\bigg).
\end{equation}
By \eqref{eq22.-2} we have
\begin{equation}\label{eq22.25}
\min\bigg(\frac{4}{\alpha}|A|,\frac{4}{\alpha}|B|\bigg) \geq 4\max(|A|,|B|) >
    |A|+|B|-1+\varepsilon\min(|A|,|B|),
\end{equation}
and by \eqref{eq22.-2} and \eqref{eq22.1} we find that
\begin{equation}\label{eq22.27}
    2\leq \min\big(|A^*|,|B^*|\big), \ \ \ \frac{\alpha}{2} |B^*| \leq |A^*| \leq \frac{2}{\alpha}|B^*| \ \ \ {\rm and} \ \ \  |A^*|+|B^*| \leq (1-\beta)p.
\end{equation}
Hence \eqref{eq22.-1}, \eqref{eq22.2} and \eqref{eq22.25} imply
\begin{equation}\label{eq22.28}
    |A|+|B|-1+\varepsilon\min(|A|,|B|) \geq (1-\mu)|A^*+B^*|,
\end{equation}
and \eqref{eq22.1}, \eqref{eq22.27} and \eqref{eq22.28} imply
\begin{equation}\label{eq22.3}
    \begin{split}
        |A^*+B^*| &\leq (1-\mu)^{-1}(|A|+|B|-1+\varepsilon\min(|A|,|B|)) \\
        &\leq (1-\mu)^{-2}|A^*|+(1-\mu)^{-2}|B^*|-1+2\varepsilon \min(|A^*|, |B^*|)\\
        &\leq |A^*|+|B^*|-1+8\mu\max(|A^*|,|B^*|)+ 2\varepsilon\min(|A^*|,|B^*|)\\
        &\leq |A^*|+|B^*|-1+4\varepsilon\min(|A^*|,|B^*|).
        \end{split}
\end{equation}
Theorem~\ref{Freiman_Zp}, together with \eqref{eq22.27} and \eqref{eq22.3}, now guarantees that there exist arithmetic progressions $U$ and $V$ with the same common difference such that
\begin{equation}\label{eq22.4}
    |A^*\Delta U| \leq \frac{\gamma}{2^{11}}\min(|A^*|,|B^*|) \text{ and } |B^*\Delta V| \leq \frac{\gamma}{2^{11}}\min(|A^*|,|B^*|)
\end{equation}
Claim A follows from  \eqref{eq22.1} and \eqref{eq22.4}:
\begin{eqnarray*}
    &&|A\Delta U|\leq |A^*\Delta U|+|A\Delta A^*| \leq \frac{\gamma}{2^{11}}\min(|A^*|,|B^*|)+\mu|A|\\
    &&\leq \frac{\gamma}{2^{11}}\min(|A|,|B|)+\mu|A|
    \leq \frac{\gamma}{2^{10}}\min(|A|,|B|).
\end{eqnarray*}
and
\begin{eqnarray*}
    &&|B\Delta V|\leq |B^*\Delta V|+|B\Delta B^*|
    \leq \frac{\gamma}{2^{11}}\min(|A^*|,|B^*|)+\mu|B|\\
    &&\leq \frac{\gamma}{2^{11}}\min(|A|,|B|)+\mu|B|
    \leq \frac{\gamma}{2^{10}}\min(|A|,|B|).  \hspace{195pt} \hfill{\square}
\end{eqnarray*}

\vspace{5pt}
\begin{equation}\label{eq22.45}
{\bf Claim B.} \ \ \ \ \ \ \ \ \ \     \E_{B' \in B^{(c)}}|( A\cap U)+(B'\cap V)|\geq |A|+|B|-1-\frac{\gamma}{2^8}\min(|A|,|B|). \ \ \ \ \ \ \ \ \ \ \ \ \ \ \ \ \ \ \ \ \ \ \
\end{equation}

\noindent
{\em Proof of Claim B.}
By Claim A and Lemma~\ref{change_c} we have
\begin{equation}\label{eq22.5}
    \E_{B'\in B^{(c)}} |(A\cap U)+(B'\cap V)| \geq (1-\mu)\E_{C'\in (B\cap V)^{(c/2)}}| (A\cap U)+C'|
\end{equation}
By Theorem~\ref{implication_shao1-prov}, we may find subsets $A^*\subset A\cap U$ and $B^*\subset B\cap V$ with
\begin{equation}\label{eq22.6}
    |A^*|\geq \lceil (1-\mu)|A\cap U|\rceil  \text{ and } |B^*|\geq \lceil (1-\mu)|B\cap V|\rceil
\end{equation}
such that
\begin{equation}\label{eq22.7}
    {\mathbb E}_{C'\in (B\cap V)^{(c/2)}}| (A\cap U)+C'| \geq \min \bigg((1-\mu)|A^*+B^*|\text{, } \frac{4}{\alpha}|A\cap U|\text{, } \frac{4}{\alpha}|B\cap V|\bigg).
\end{equation}
By \eqref{eq22.-2} and \eqref{eq22.05} we have
\begin{equation}\label{eq22.8}
\min\bigg(\frac{4}{\alpha}|A\cap U|,\frac{4}{\alpha}|B\cap V|\bigg) \geq 2\max(|A\cap U|,|B\cap V|) \geq |A^*|+|B^*|.
\end{equation}
By the Cauchy--Davenport theorem, together with \eqref{eq22.-2}, we have
\begin{equation}\label{eq22.10}
    |A^*+B^*|\geq \min(p,|A^*|+|B^*|-1)=|A^*|+|B^*|-1.
\end{equation}
By \eqref{eq22.5}, \eqref{eq22.7}, \eqref{eq22.8} and \eqref{eq22.10} this yields
\begin{equation}\label{eq22.9}
    \E_{B'\in B^{(c)}} |(A\cap U)+(B'\cap V)| \geq (1-\mu)^2(|A^*|+|B^*|-1).
\end{equation}
Finally, from \eqref{eq22.05}, \eqref{eq22.6} and \eqref{eq22.9} we deduce that
\begin{equation*}
\begin{split}
    \E_{B'\in B^{(c)}} |(A\cap U)+(B'\cap V)|  &\geq (1-\mu)^2(|A^*|+|B^*|-1) \\
    &\geq (1-\mu)^3(|A\cap U|+|B\cap V|)-1\\
    &\geq (1-\mu)^3(|A|+|B|-2^{-9}\gamma\min(|A|,|B|))-1\\
    &\geq |A|+|B|-1-2^{-9}\gamma \min(|A|,|B|)-8\mu \max(|A|,|B|)\\
    &\geq |A|+|B|-1-2^{-8}\gamma \min(|A|,|B|). \hspace{155pt}\square
\end{split}
\end{equation*}

\vspace{5pt}
Let $Q$ be the arithmetic progression which extends $V$ on each side by exactly $\lfloor 2^{-2}\gamma \min(|A|,|B|) \rfloor$.\\

\vspace{5pt}
\[
{\bf Claim\ C.} \ \ \ \ \ \ \ \ \ \ \ \ \ \ \ \   |B \Delta Q| \leq \gamma \min(|A|,|B|) \text{ and } |B\setminus Q|=0. \ \ \ \ \ \ \ \ \ \ \ \ \ \ \ \ \ \ \ \ \ \ \ \ \ \ \ \ \ \ \ \
\]

\noindent
{\em Proof of Claim C.}
For the first part, by \eqref{eq22.05} we have
\begin{equation}\label{eq22.105}
\begin{split}
        |B\Delta Q|&\leq |B\Delta V|+|V\Delta Q|
        \leq 2^{-10}\gamma \min(|A|,|B|)+ 2\lfloor 2^{-2}\gamma \min(|A|,|B|)\rfloor\\
        &\leq \gamma \min(|A|,|B|).
\end{split}
\end{equation}
For the second part, we assume for a contradiction that we have some
$b\in B\setminus Q$.
This implies $\lfloor2^{-2}\gamma \min(|A|,|B|)\rfloor\geq 1$,
which gives
\begin{equation}\label{eq22.11}
    \lfloor 2^{-2}\gamma \min(|A|,|B|) \rfloor \geq  2^{-3}\gamma \min(|A|,|B|).
\end{equation}
By \eqref{eq22.-2}, \eqref{eq22.05} and \eqref{eq22.105} we have
$|U|+|Q| \leq (1+\gamma)(|A|+|B|)\le (1+\gamma)(1-\beta)p <p$ and
$|U| \geq 2^{-1}|A| \geq \lfloor 2^{-2} \gamma \min(|A|,|B|) \rfloor$, which yields
\begin{equation}\label{eq22.12}
    |(U+b)\setminus (U+V)|\geq \lfloor 2^{-2}\gamma \min(|A|,|B|) \rfloor.
\end{equation}
By \eqref{eq22.05}, \eqref{eq22.11}, \eqref{eq22.12} we have
\begin{eqnarray}\label{eq22.13}
        |(A+b) \setminus (U+V)|&\geq& -|A\Delta U|+|(U+b)\setminus (U+V)|\\
        &\geq& -2^{-10}\gamma \min(|A|,|B|)+ \lfloor 2^{-2}\gamma \min(|A|,|B|) \rfloor
        \geq 2^{-4}\gamma \min(|A|,|B|).
\end{eqnarray}
Finally, by \eqref{eq22.-1}, \eqref{eq22.45} and \eqref{eq22.13} we obtain
\begin{eqnarray*}
        |A|+|B|-1 + \varepsilon \min(|A|,|B|) &\geq& \E_{B'\in B^{(c)}}\bigg|A+(B'\cup \{b\})\bigg| \\
        &\geq& \bigg|(A+b)\setminus (U+V)\bigg|+\E_{B'\in B^{(c)}}\bigg|(A\cap U)+(B'\cap V))\bigg|\\
        &\geq& 2^{-4}\gamma \min(|A|,|B|)+ |A|+|B|-1-2^{-8}\gamma\min(|A|,|B|)\\
        &\geq& |A|+|B|-1+2^{-5}\gamma \min(|A|,|B|). \hspace{100pt}{\square}
\end{eqnarray*}

\vspace{7pt}
\noindent
The conclusion of Theorem~\ref{weak_stab_Zp} now follows from Claims A and C.
\end{proof}

\vspace{5pt}
We need one final result.
\begin{theorem}\label{CD_technical}
  For all $\beta>0$ there exists $\gamma>0$ such that for every $\alpha>0$ there is a value of $c$ for which the
  following
  holds. Let $A$ and $B$ be subsets of $\mathbb{Z}_p$ and let  $I=[p_l, p_r]$ and $J=[q_l, q_r]$ be intervals  of $\mathbb{Z}_p$ satisfying
\begin{equation}\label{eq25.-2}
    \alpha |J| \leq |I| \leq \alpha^{-1} |J| \text{ and } |I|+|J|\leq (1-\beta)p
\end{equation}
and
\begin{equation}\label{eq25.-1}
    \max(|A\Delta I|, |B\Delta J|) \leq \gamma \min (|I|, |J|) \text{ and } \{q_l, q_r\} \subset B \subset J.
\end{equation}
Then there is a family $\mathcal{F}\subset B^{(c)}$, depending only on $I$, $J$ and $B$ (but not on $A$), such that
$$\E_{B'\in \mathcal{F}}|A+B'|\geq |A|+|J|-1 \geq |A|+|B|-1.$$
\end{theorem}

\begin{proof}
Fix $2^{-10}>\beta>0$. Let $\gamma=2^{-1000}\beta$. Fix $2^{-10}>\alpha>0$. Let $s=\lfloor 2^{10} \log (\alpha^{-1})\rfloor$. Let $c_1=(18^{13}\alpha^{-1}+1)(s+1)$. Let $c_0$ be the output of Lemma~\ref{change_c_again} with input $m=\lfloor 4\beta^{-1}\rfloor$ and $\mu=2^{-5}\alpha \beta$.  Let $c'=\max(c_0,c_1)$ and $c=5c'+2$. The proof will be divided into three lemmas.\\

First, note that if $A=I$ and $B=J$ then trivially there exists a set $B_0 \subset B$ (depending only on $I$ and $J$) such that
$|A+B_0| \geq |A|+|J|-1 \text{ and } |B_0| \leq c'$.
Hence, for the rest of the proof we may assume that $A \neq I$ or $B \neq J$, and so
\begin{equation}\label{eq25.0}
\min(|I|,|J|) \ge 1/\gamma \ \ \text{and} \ \ \beta \ge 2^{1000} \gamma.
\end{equation}
Moreover,  by passing to a subinterval $I'$ of $I$ and changing $\alpha$ to $2^{-1}\alpha$ and $\gamma$ to $2\gamma$ if needed, we may assume that the size of $|I|$ is a multiple of $6$.

\begin{lemma}\label{lem27}
We have
\begin{equation*}\label{eq25.01}
    \E_{B'\in B^{(c')}}\bigg|\bigg((A\cap I^c)+(B'\cup \{q_l,q_r\})\bigg)\cap \bigg(I+J\bigg)^c\bigg|\geq |A\cap I^c|.
\end{equation*}
\end{lemma}
\begin{proof}
Construct intervals $I_l$ and $I_r$ of size $\lfloor 2^{-1}\beta \min(|I|,|J|)  \rfloor$ to the left and right of the interval $I$: \ $I_l=p_l-[ 2^{-1}\beta \min(|I|,|J|) ] \text{ and } I_r= p_r+[ 2^{-1}\beta \min(|I|,|J|) ]$.
By \eqref{eq25.-2}, we can construct the disjoint unions
$K=I_l\sqcup I \sqcup I_r$
and $L=(I_l+q_l)\sqcup (I+J) \sqcup (I_r+q_r)$.
By \eqref{eq25.-2}, there exists a collection $\mathcal{J}$ of $\lfloor 4\beta^{-1} \rfloor$ intervals
$\mathcal{J}=\{J_1, \hdots, J_{\lfloor 4\beta^{-1} \rfloor}\}$
such that for every $J_j\in {\mathcal J}$ we have $J_j\subset J$ \ and \  $|J_j|= \lfloor 2^{-1}\beta \min(|I|,|J|)  \rfloor \geq 2^{10}$, \
and for every $x \in K^c$ there exists $j_x \in [ 4\beta^{-1}]$ such that
\begin{equation}\label{eq25.5}
(x+J_{j_x})\cap (I+J)=\emptyset.
\end{equation}
For $J_j \in {\mathcal J}$ set $B_j=J_j\cap B$.
Recalling \eqref{eq25.-2} and \eqref{eq25.-1}, we find that
\begin{equation}\label{eq25.55}
|B_j|\geq 2^{-3}\beta \min(|I|,|J|) \geq 2^{-5}\alpha \beta |B|.
\end{equation}

\vspace{10pt}
It is time to put together what we have learned about our sets. By construction we have
\begin{equation}\label{eq25.6}
\E_{B'\in B^{(c')}}\bigg|\bigg((A\cap I^c)+(B'\cup \{q_l,q_r\})\bigg)\cap \bigg(I+J\bigg)^c\bigg|
\end{equation}
\begin{equation*}\label{eq25.7}
\ge \bigg|A\cap I_l \bigg|+ \bigg|A\cap I_r\bigg|+ \E_{B'\in B^{(c')}}\bigg|\bigg((A\cap I^c)+B'\bigg)\setminus \bigg(\big[(A\cap I_l)+q_l\big]\cup \big[I+J\big]\cup \big[(A\cap I_r)+q_r\big]\bigg)\bigg|.
\end{equation*}
By Lemma~\ref{change_c_again} combined with \eqref{eq25.55} we have
\begin{equation}\label{eq25.8}
\E_{B'\in B^{(c')}}\bigg|\bigg((A\cap I^c)+B'\bigg)\setminus \bigg(\big[(A\cap I_l)+q_l\big]\cup \big[I+J\big]\cup \big[(A\cap I_r)+q_r\big]\bigg)\bigg|
\end{equation}
\begin{equation*}
\ge \frac{3}{4}\E_{\substack{b_j^1,b_j^2\in B_j \\ j \in [4\beta^{-1}]} }\bigg|\bigg((A\cap I^c)+\cup_j\{b_j^1,b_j^2\}\bigg)\setminus \bigg(\big[(A\cap I_l)+q_l\big]\cup \big[I+J\big]\cup \big[(A\cap I_r)+q_r\big]\bigg)\bigg|.
\end{equation*}
By \eqref{eq25.5} we have
\begin{equation}\label{eq25.10}
\E_{\substack{b_j^1,b_j^2\in B_j \\ j\in[4\beta^{-1}]} }\bigg|\bigg((A\cap K^c)+\cup_j\{b_j^1,b_j^2\}\bigg)\setminus \bigg(\big[(A\cap I_l)+q_l\big]\cup \big[I+J\big]\cup \big[(A\cap I_r)+q_r\big]\bigg)\bigg|
\end{equation}
\begin{equation*}\label{eq25.115}
\ge \E_{\substack{b_j^1,b_j^2\in B_j\\ j\in[4\beta^{-1}]} }\bigg|\bigg(\bigcup_{a\in A\cap K^c}a+\{b_{j_a}^1,b_{j_a}^2\}\bigg)\setminus \bigg(\big[(A\cap I_l)+q_l\big]\cup \big[(A\cap I_r)+q_r\big]\bigg)\bigg|.
\end{equation*}
A trivial bound gives
\begin{equation}\label{eq25.117}
\E_{\substack{b_j^1,b_j^2\in B_j \\ j\in[4\beta^{-1}]} }\bigg|\bigg(\bigcup_{a\in A\cap K^c}a+\{b_{j_a}^1,b_{j_a}^2\}\bigg)\setminus \bigg(\big[(A\cap I_l)+q_l\big]\cup \big[(A\cap I_r)+q_r\big]\bigg)\bigg|
\end{equation}
\begin{equation*}\label{eq25.13}
\ge \sum_{\substack{a \in A\cap K^c \\ i \in \{1,2\}}}\E_{\substack{b_j^1,b_j^2\in B_j\\ j\in[4\beta^{-1}]} }\bigg|\bigg(a+b_{j_a}^i\bigg)\setminus \bigg(\big[(A\cap I_l)+q_l\big]\cup \big[(A\cap I_r)+q_r\big] \cup \big[ \bigcup_{\substack{a' \in A\cap K^c \\ i' \in \{1,2\}: \\ j_{a'}\neq j_a \text{ or } i' \neq i }} a'+b_{j_{a'}}^{i'} \big] \bigg)\bigg|.
\end{equation*}

\vspace{10pt}
Continuing, by \eqref{eq25.-1}, given  $b_j^i\in B_j$ for  $i=1,2$ and $j \in [4\beta^{-1}]$, we have
\begin{equation*}\label{eq25.14}
    \bigg|  \big[(A\cap I_l)+q_l\big]\cup \big[(A\cap I_r)+q_r\big] \cup \big[ \bigcup_{\substack{a \in A\cap K^c \\ i \in \{1,2\} }} a+b_{j_{a}}^{i} \big] \bigg| \leq 2|A \cap I^c| \leq 2\gamma \min(|I|, |J|).
\end{equation*}
By \eqref{eq25.55} and \eqref{eq25.14} we have
\begin{equation}\label{eq25.15}
\sum_{\substack{a \in A\cap K^c \\ i \in \{1,2\}}}\E_{\substack{b_j^1,b_j^2\in B_j \\ j\in[4\beta^{-1}]} }\bigg|\bigg(a+b_{j_a}^i\bigg)\setminus \bigg(\big[(A\cap I_l)+q_l\big]\cup \big[(A\cap I_r)+q_r\big] \cup \big[ \bigcup_{\substack{a' \in A\cap K^c \\ i' \in \{1,2\}: \\ j_{a'}\neq j_a \text{ or } i' \neq i }} a'+b_{j_{a'}}^{i'} \big] \bigg)\bigg|
\end{equation}
\begin{equation*}\label{eq25.17}
\ge \frac{3}{4}2|A\cap K^c| = \frac{3}{2}|A\cap K^c|.
\end{equation*}
By \eqref{eq25.8}, \eqref{eq25.10}, \eqref{eq25.117} and \eqref{eq25.15} we get
\begin{equation}\label{eq25.18}
\E_{B'\in B^{(c')}}\bigg|\bigg((A\cap I^c)+B'\bigg)\setminus \bigg(\big[(A\cap I_l)+q_l\big]\cup \big[I+J\big]\cup \big[(A\cap I_r)+q_r\big]\bigg)\bigg|\geq \frac{9}{8}|A\cap K^c|.
\end{equation}
Finally, by \eqref{eq25.6} and \eqref{eq25.18} we obtain
\begin{equation*}\label{eq25.19}
\E_{B'\in B^{(c')}}\bigg|\bigg((A\cap I^c)+(B'\cup \{q_l,q_r\})\bigg)\cap \bigg(I+J\bigg)^c\bigg|\geq |A\cap I_l|+|A\cap I_r|+\frac{9}{8}|A\cap K^c|\geq |A\cap I^c|.
\end{equation*}

\vspace{5pt}
This concludes the proof of Lemma~\ref{lem27}.
\end{proof}

In the rest of the proof of Theorem~\ref{CD_technical} we shall only be interested in the subsets $A\cap I$ and $B=B\cap J$. Given that the intervals $I$ and $J$ satisfy $|I|+|J|<p$, we may assume from now on that the ambient space is $\mathbb{Z}$ rather than $\mathbb{Z}_p$.

\vspace{10pt}
Here is the second lemma we shall need.

\begin{lemma}\label{lem25}
For each $i=1,2,3$ there exists a parameter $d_i \in \mathbb{N}$ and there exist consecutive intervals $I_i=[p_{i-1}+1, p_i] \text{ and } J_i=[q_{i-1}, q_i]$,
depending only on $I$, $J$ and $B$ (but not on $A$), such that the sets $A_i=A\cap I_i \text{ and } B_i=B\cap J_i$
satisfy

{\rm (1)} \hspace{92pt} $I=I_1\sqcup I_2\sqcup I_3 \ \text{and }\ J=J_1 \cup J_2 \cup J_3$,

\vspace{8pt}
{\rm (2)} \hspace{115pt}     $6^{-13}\alpha |J_i| \leq |I_i| \leq 6^{13}\alpha^{-1} |J_i|$,

{\rm (3)} \hspace{75pt}  $\max\bigg(|I_i\Delta A_i|, |J_i\Delta B_i|\bigg) \leq 6^{13}\gamma \min (|I_i|, |J_i|)$,

{\rm (4)} \hspace{105pt}  $3d_i \leq \min (|I_i|, |J_i|) \leq (3^{10}+1)d_i$,

\vspace{8pt}
{\rm (5)} \hspace{150pt}  $d_i|(q_i-q_{i-1})$,

\vspace{8pt}
{\rm (6)} \hspace{97pt} $J_i\cap \{q_i \ \text{\rm mod }\ d_i \} = B_i\cap \{q_i \ \text{\rm mod }\  d_i \}$.

\end{lemma}

\begin{proof}
Fix parameters $k_{1}=2^{10} \text{ and } k_3=3^{10}$
and
\begin{equation}\label{eq25.249}
k_2=6\lfloor \frac{|J|}{\min(|I|,|J|)} \rfloor-1 \text{ and } d_2=\lfloor6^{-1}\min(|I|,|J|)\rfloor.
\end{equation}
Since \eqref{eq25.0} holds and  $|I|$ is a multiple of 6,  we have
\begin{equation}\label{eq25.2495}
|J|-6^{-1}\min(|I|,|J|)\geq k_2d_2\geq \max\bigg(|J|-2\min(|I|,|J|),2^{-1}\min(|I|,|J|)\bigg).
\end{equation}
By the Chinese remainder theorem, there exist parameters $q_{l}, q_{r}, d_1, d_3 \in \mathbb{Z}$ such that
$q_1=q_l+k_1d_1$, \   $q_{2}=q_r-k_3d_3$ \  and \ $q_2-q_1=k_2d_2$.
For $t\in \mathbb{Z}$ construct parameters \
$q_{0,t}=q_l$, \ $q_{1,t}=q_1+tk_1k_3$, \  $q_{2,t}= q_2+tk_1k_3$ \  and \ $q_{3,t}=q_r$,
and
$d_{1,t}=d_1+tk_3$, \  $d_{2,t}=d_2$ \ and \ $d_{3,t}=d_3-tk_1$
such that
$q_{1,t}=q_{0,t}+k_1d_{1,t}$, \   $q_{2,t}=q_{3,t}-k_3d_{3,t}$ \ and \ $q_{2,t}-q_{1,t}=k_{2}d_{2,t}=k_2d_2$.

Then the set $T=\{t\in \mathbb{Z} \text{ : }q_l=q_{0,t} <q_{1,t}<q_{2,t}<q_{3,t}=q_r \}$ is an interval of size
$|T|\geq \frac{|J|-2-k_2d_2}{k_1k_3}-1$.
Using our earlier inequalities, we find that $|T|\geq 6^{-12}\min(|I|,|J|)\geq 6$.
Therefore, for  $i=1,2,3$ there are disjoint consecutive intervals
$T_i=[t_{i-1}+1, t_i]$ of sizes satisfying  $|T_2|=\lfloor 6^{-13}\min(|I|,|J|) \rfloor \geq 1$ and
$\min(|T_1|,|T_3|)\geq 6^{-13}\min(|I|,|J|)$,
which form a partition $T=T_1\sqcup T_2\sqcup T_3$.
We also have $|I|\geq 4\lfloor 2^{-2}\min(|I|,|J|) \rfloor \geq 4$.
Therefore, for  $i=1,2,3$ there are disjoint consecutive intervals $I_i=[p_{i-1}+1,p_i]$ of sizes
$|I|\geq |I_2|\geq 2^{-1}|I|$ \ and \ $|I_1|=|I_3|=\lfloor 2^{-2} \min(|I|,|J|) \rfloor\geq 1$,
forming a partition \ $I=I_1\sqcup I_2 \sqcup I_3$.
Finally, for $i=1,2,3$ and $t\in T_2$  we take the overlapping consecutive intervals
$J_{i,t}=[q_{i-1,t},q_{i,t}],$
which form a cover \ $J=J_{1,t}\cup J_{2,t}\cup J_{3,t}$.

Now define functions                                                                                                                                                                   
\begin{equation}\label{eq25.29963}                                                                                                                                                 
g_1(i_1,t)=q_{0,t}+i_1 d_{1,t} \text{, } \ \ \  g_2(i_2,t)=q_{1,t}+i_2 d_{2,t} \ \ \  \text{and} \ \ \  g_3(i_3,t)=q_{3,t}-i_3 d_{3,t}.                                            
\end{equation}                                 
                                                                                                                                                                                   
\vspace{5pt}                                                                                                                                                                       
\noindent                                                                                                                                                                          
Form the partition $T_2=S\sqcup S^c$, where                                                                                                                                        
\begin{equation}\label{eq25.29965}                                                                                                                                                 
S=\{t\in T_2 : \ \  g_i(j_i,t)\in B \ \ \text{for all} \ \ i=1,2,3 \ \ \ and \ \ j_i \in [k_i]\}.                                                                                         
\end{equation}

It turns out that $S$ is non-empty, and that if we set, for any fixed $t \in S$, $J_i = J_{i,t}$, then the given
parameters and intervals do have the required properties (1)-(6).  
The proof of this is rather lengthy and tedious, and not very instructive. For this reason we
present the details in the Appendix.
\end{proof}

\vspace{5pt}
The final lemma we shall need in the proof of Theorem~\ref{CD_technical} is as follows.
\begin{lemma}\label{lem30}
Under the conditions of Lemma~\ref{lem25}, for $i=1,2,3$ we have
\[
\E_{\substack{w_j \in \mathbb{Z}_{d_i} \\ j\in[s] }} \bigg|A_i+(B_i^{q_i}\cup_{j\in [s]} B_i^{w_j}) \bigg| \geq |A_i|+|J_i|-1 \ \ \ {\text and} \ \ \
\max_{\substack{w_j \in \mathbb{Z}_{d_i} \\ j\in[s] }} \bigg|B_i^{q_i}\cup_{j\in [s]} B_i^{w_j} \bigg| \leq c'.
\]
\end{lemma}
\begin{proof}
Fix $i$. The proof is divided into two claims, but before we begin we need a simple remark. By condition $(5)$ we can write $q_i-q_{i-1}=k_id_i$ for some $k_i\in \mathbb{N}$. Moreover, by conditions $(2)$ and $(4)$ we have $k_i \leq 18^{13}\alpha^{-1}$. \\

\vspace{7pt}
{\bf Claim A. }{\em We have $\bigg|A_i+B_i^{q_i}\bigg| \geq |A_i|+k_i|\widetilde{A_i}|$
and
$\pi_{d_i}\bigg(A_i+B_i^{q_i}\bigg)=q_i+\widetilde{A_i}$.
}
\begin{proof}
By condition $(6)$  we have $|B_i^{q_i}|=k_i+1$. Also, if $x\in q_i+ \widetilde{A_i}$   then $|A_i^{x-q_i}| \geq 1$.
By Theorem~\ref{3k-4} it follows that $\bigg| (A_i+B_i^{q_i})^x\bigg|= |A_i^{x-q_i}|+|B_i^{q_i}|-1 =|A_i^{x-q_i}|+k_i$.
Hence
\begin{equation*}\label{eq30.6}
\bigg|A_i+B_i^{q_i}\bigg| \geq \sum_{x\in q_i+\widetilde{A}} \bigg|(A_i+B_i^{q_i})^x\bigg| \geq \sum_{y\in \widetilde{A_i}} |A_i^y|+k_i = |A_i|+k_i|\widetilde{A_i}|
\end{equation*}
completing the proof of Claim A.
\end{proof}

\vspace{7pt}
{\bf Claim B. }{\em   For  $z\in \mathbb{Z}_{d_i}$, we have
$\E_{\substack{w_j \in \mathbb{Z}_{d_i} \\ j\in[s] }} \bigg|(A_i+\cup_{j\in [s]} B_i^{w_j})^z \bigg| \geq k_i$. }

\begin{proof}
By condition (3) we have
\begin{equation*}\label{eq30.8}
\max\bigg(\sum_{w\in \mathbb{Z}_{d_i}} |I_i^w \setminus A_i^w|,  \sum_{w\in \mathbb{Z}_{d_i}} |J_i^w \setminus B_i^w| \bigg) \leq \max \bigg(|I_i\Delta A_i| , |J_i\Delta B_i| \bigg) \leq 6^{13}\gamma \min(|I_i|,|J_i|),
\end{equation*}
and condition (4) implies that
\begin{equation*}\label{eq30.9}
\max\bigg(\sum_{w\in \mathbb{Z}_{d_i}} |I_i^w \setminus A_i^w|,  \sum_{w\in \mathbb{Z}_{d_i}} |J_i^w \setminus B_i^w| \bigg) \leq 18^{13}\gamma d_i.
\end{equation*}
By condition $(4)$, if $w\in \mathbb{Z}_{d_i}$ then
$|I_i^w| \geq 2 \text{ and } |J_i^w| \geq k_i$, so
\begin{equation*}\label{eq30.11}
\max\bigg(\sum_{w\in \mathbb{Z}_{d_i}} \max(2-|A_i^w|,0),  \sum_{w\in \mathbb{Z}_{d_i}} \max(k_i- |B_i^w|, 0) \bigg) \leq 18^{13}\gamma d_i.
\end{equation*}
We conclude that the sets
$X=\{w\in \mathbb{Z}_{d_i} \text{ : } |A_i^{w}| \leq 1\}$ \  and  \  $Y=\{w\in \mathbb{Z}_{d_i} \text{ : } |B_i^{w}| \leq k_i-1\}$
are not too large: $\max(|X|,|Y|) \leq 18^{13}\gamma d_i \leq 3^{-1}d_i$.
Thus
$\Prob_{w\in \mathbb{Z}_{d_i}}\bigg( w\not \in Y \text{ and } z-w \not \in X \bigg) \geq 1/3$
and
$\Prob_{\substack{w_j\in \mathbb{Z}_{d_i} \\ j\in [s]}}\bigg( \exists j \text{ : } w_j\not \in Y \text{ and } z-w_j \not \in X \bigg) \geq 1-2^s3^{-s}$.
Equivalently, the event
\[
E=\bigg\{(w_1,\hdots,w_d) \in \mathbb{Z}_{d_i}^s \text{ : } \exists j \text{ : } w_j\not \in Y \text{ and } z-w_j \not \in Y \bigg\}
\]
has measure
\[
\Prob_{\substack{w_j\in \mathbb{Z}_{d_i} \\ j\in [s]}}\bigg((w_1, \hdots, w_s)\in E\bigg)\geq 1-2^s3^{-s}.
\]
By construction and by Theorem~\ref{3k-4}, if $(w_1,\hdots, w_s) \in E \subset \mathbb{Z}_{d_i}^s$
then
\begin{equation*}\label{eq30.19}
\bigg| (A_i+\cup_{j\in [s]} B_i^{w_j})^z\bigg| \geq \max_{j\in [s]} \bigg|A_i^{z-w_j}+B_i^{w_j} \bigg| \geq 2+k_i-1=k_i+1.
\end{equation*}
From the last two inequalities it follows that
\begin{equation*}\label{eq30.20}
\begin{split}
    \E_{\substack{w_j \in \mathbb{Z}_{d_i} \\ j\in[s] }} \bigg|(A_i+\cup_{j\in [s]} B_i^{w_j})^z \bigg| &\geq \E_{\substack{w_j \in \mathbb{Z}_{d_i} \\ j\in[s] }} \bigg( \bigg|(A_i+\cup_{j\in [s]} B_i^{w_j})^z \bigg| \text{  } \text{  } \text{ : } \text{  } \text{  } E\bigg) \Prob_{\substack{w_j\in \mathbb{Z}_{d_i} \\ j\in [s]}}\bigg((w_1, \hdots, w_s)\in E\bigg)\\
    &\geq (k_i+1)(1-2^{s}3^{-s}) \geq k_i,
\end{split}
\end{equation*}
where the last inequality follows from the estimates $k_i \leq 18^{13}\alpha^{-1}$, $s\geq 2^9\log \alpha^{-1}$ and $\alpha <2^{-{10}}$.
This finishes the proof of Claim B.
\end{proof}

Returning to the proof of Lemma~\ref{lem30}, we see that by
combining Claim A and Claim B we obtain the first part of what we want:
\begin{eqnarray*}
    \E_{\substack{w_j \in \mathbb{Z}_{d_i} \\ j\in[s] }} \bigg|A_i+(B_i^{q_i}\cup_{j\in [s]} B_i^{w_j}) \bigg| &\geq&     \sum_{z\in \mathbb{Z}_{d_i}}\E_{\substack{w_j \in \mathbb{Z}_{d_i} \\ j\in[s] }} \bigg|\bigg(A_i+(B_i^{q_i}\cup_{j\in [s]} B_i^{w_j})\bigg)^z \bigg|\\
    &\geq& \sum_{z\in q_i+\widetilde{A_i}} \bigg|(A_i+B_i^{q_i})^z \bigg|+\sum_{z\in (q_i+\widetilde{A_i})^c}\E_{\substack{w_j \in \mathbb{Z}_{d_i} \\ j\in[s] }} \bigg|(A_i+\cup_{j\in [s]} B_i^{w_j})^z \bigg|\\
    &\geq& \bigg|A_i+B_i^{q_i} \bigg|+\sum_{z\in (q_i+\widetilde{A_i})^c}\E_{\substack{w_j \in \mathbb{Z}_{d_i} \\ j\in[s] }} \bigg|(A_i+\cup_{j\in [s]} B_i^{w_j})^z \bigg|\\
    &\geq& |A_i|+|\widetilde{A_i}|k_i+(d_i-|\widetilde{A_i}|)k_i = |A_i|+d_ik_i=|A_i|+|J_i|-1.
\end{eqnarray*}

\vspace{3pt}
\noindent
Note also that for each $w\in \mathbb{Z}_{d_i}$, by construction we have
$|B_i^w| \leq k_i+1.$
Hence we have the second part
\begin{equation*}
    \max_{\substack{w_j \in \mathbb{Z}_{d_i} \\ j\in[s] }} \bigg|B_i^{q_i}\cup_{j\in [s]} B_i^{w_j} \bigg| \leq (s+1)(k_i+1) \leq c',
\end{equation*}
where last inequality follows from the bounds $k_i \leq 18^{13}\alpha^{-1}$, $s\leq 2^{10}\log \alpha^{-1}$ and $\alpha <2^{-{10}}$.
This concludes the proof of Lemma~\ref{lem30}.
\end{proof}

\vspace{5pt}
To finish the proof of Theorem~\ref{CD_technical} we shall construct the desired family of sets $\mathcal{F}$. We remark that we allow these families to have repeated sets. For each $i=1,2,3$ define
$\mathcal{F}_i=\{B_i^{q_i}\cup_{j\in [s]} B_i^{w_j} \text{ : } (w_1, \hdots, w_s) \in \mathbb{Z}^s_{d_i} \}.$
In addition, put
$\mathcal{F}_0=\{B_0\} \text{ and }\mathcal{F}_4=\{B'\cup\{q_l,q_r\} \text{ : } B'\in B^{(c')}\}.$
Finally, let
$\mathcal{F}=\{F_0\cup F_1\cup F_2\cup F_3\cup F_4 \text{ : } F_i\in \mathcal{F}_i\}.$
By construction (see  Lemmas~\ref{lem25} and \ref{lem30}, and the remark before Lemma~\ref{lem27}), the family $\mathcal{F}$ depends only on the sets $I$, $J$ and $B$ (but not $A$), and  any set $F\in \mathcal{F}$ has size
$|F| \leq 3c'+c'+c'+2\leq c. $
In the case $A=I$ and $B=J$,  we have
$\E_{B'\in \mathcal{F}}\big|A+B'\big| \geq   \big|A+B_0\big| \geq |A|+|J|-1.$

\vspace{5pt}
In the case $A\neq I$ or $B\neq J$, combining Lemmas~\ref{lem27} and \ref{lem30}, we obtain
\begin{equation*}
    \begin{split}
        \E_{B'\in \mathcal{F}}\bigg|A+B'\bigg| &\geq   \E_{\substack{B_i'\in \mathcal{F}_i \\ i \in [4]}}\bigg|A+\cup_{i\in[4]}B_i'\bigg|  \\
        &\geq \E_{B'\in B^{(c')}}\bigg|\bigg(A+(B'\cup \{q_l,q_r\})\bigg)\cap \bigg(I+J\bigg)^c\bigg|+\sum_{i\in [3]}\E_{B_i'\in \mathcal{F}_i}\bigg|(A+B_i')\cap (I_i+J_i) \bigg|\\
        &\geq \E_{B'\in B^{(c')}}\bigg|\bigg((A\cap I^c)+(B'\cup \{q_l,q_r\})\bigg) \cap \bigg(I+J \bigg)^c\bigg|+\sum_{i\in [3]}\E_{B_i'\in \mathcal{F}_i}\bigg|A_i+B_i' \bigg|\\
        &\geq |A\cap I^c|+ \sum_{i\in[3]}|A_i|+|J_i|-1
        \geq |A\cap I^c|+|A\cap I|+|J|-1
        \geq |A|+|J|-1,
    \end{split}
\end{equation*}
where the second inequality follows by considering the partition $\mathbb{Z}_p=(I+J)^c\sqcup_{i\in [3]} (I_i+J_i)$.
This finishes the proof of Theorem~\ref{CD_technical}.
\end{proof}

\vspace{5pt}
We are now ready to prove Theorem~\ref{strong_stab_B_Z_p}.

\begin{proof}[Proof of Theorem \ref{strong_stab_B_Z_p}]
  Fix $2^{-10} >\beta>0$. Let $\gamma$ be the output of Theorem~\ref{CD_technical} with input $\beta$, with the
  extra condition that $2^{-10} \beta \geq \gamma$. Let $\varepsilon$ be the output of Theorem~\ref{weak_stab_Zp} with input $\beta,\gamma$. Let $c_1$ be the output of Theorem~\ref{weak_stab_Zp} with input $\beta, \gamma, \alpha$. Let $c_2$ be the output of Theorem~\ref{CD_technical} with input $2^{-1} \beta, 2\alpha$. Finally, let $c=\max(c_1,c_2)$. \\

\vspace{5pt}
Combining Theorem~\ref{weak_stab_Zp} with inequalities \eqref{eq34.-2} and \eqref{eq34.-1}, we see that there exist arithmetic progressions $I$ and $J$ with the same common difference such that
$|A\Delta I| \leq \gamma \min(|A|,|B|)$, \ \ $|B\Delta J| \leq \gamma \min(|A|,|B|)$ \ \ and \ \ $|B\setminus J|=0$.
By scaling, we may assume that $I$ and $J$ are intervals $I=[p_l,p_r]$ and  $J=[q_l,q_r]$.
Moreover, we may assume that \ $\{q_l,q_r\} \subset B \subset J$.
Furthermore, we have
$2^{-1}\alpha |J| \leq |I| \leq 2 \alpha^{-1}|J|$ \  and  $|I|+|J| \leq (1-2^{-1}\beta)p$.
Theorem~\ref{CD_technical} and the relations we have just found imply that
there is a family $\mathcal{F} \subset B^{(c)}$ such that $\E_{B'\in \mathcal{F}}|A+B'|\geq |A|+|J|-1$.
Finally, recalling \eqref{eq34.-1}, we obtain $|A|+|J|-1 \leq |A|+|B|-1+r.$
Thus $J$ is an arithmetic progression of size $|B|+r$ containing $B$ (and so $r \geq 0$).
\end{proof}

\vspace{5pt}
We have thus proved our main results in ${\mathbb Z}_p$ since, as we remarked at the start
of the section, Theorem \ref{strong_stab_B_Z_p} contains both Theorems~\ref{translates_of_A_Zp}
and \ref{stability_in_Zp_for_intro} as special cases.

\section{Applications -- Restricted Sums}

\vspace{5pt}
We now turn to the applications mentioned earlier.
In a slight variant of earlier notation, given sets $A, B \subset {\mathbb Z}_p$ and a set ${\mathcal F} \subset A\times B$, the {\em ${\mathcal F}$-restricted sum of $A$ and $B$} is
\[
A+_{\mathcal F} B=\{a+b: \ (a,b) \not\in {\mathcal F} \},
\]
so that ${\mathcal F}$ is the set of {\em forbidden pairs}. These general ${\mathcal F}$-restricted sums were introduced and studied by Lev~\cite{Lev-00a, Lev-00b, Lev-01} over twenty years ago. (We hope that
the reader will not be too annoyed that the subscript now refers to the pairs we do {\em not} take, as opposed to the earlier notation of $A +_{\Gamma}B$: both of these notations are standard, unfortunately,
and we will use a curly letter in this section to emphasise the difference.)

\vspace{5pt}
For $b\in B$, we call $d(b)=\big|\{a: \ (a,b)\in {\mathcal F}\}\big|$ the {\em degree} of $b\in B$, and $d({\mathcal F})=\max_{b\in B} d(b)$ the {\em degree} of ${\mathcal F}$. Note that if ${\mathcal F} \subset A \times A$ then $d(a)$ is the number of pairs $(a', a)\in {\mathcal F}$, with $a$ in the {\em second} factor of $A\times A$.

\vspace{7pt}
We start with two questions. Firstly, can one extend the Erd\H{o}s--Heilbronn conjecture along the lines of Theorem~\ref{translates_of_A_Zp_equality},
our extension of the Cauchy--Davenport theorem? In other words, is there a constant $c$ such that if $p\ge 2n-3$ and $A$ is a subset of ${\mathbb Z}_p$ of size $n$, then there is a set $B'\in A^{(\le c)}$ such that $\{a+b: \ a\in A,\ b\in B', \ a\ne b\}$ has at least $2n-3$ elements?

\vspace{7pt}
Secondly, can one extend the Erd\H{o}s--Heilbronn conjecture in another direction, replacing the sum $A+A$ by an ${\mathcal F}$-restricted sum $A+_{\mathcal F} A$ for a suitable family ${\mathcal F}$? Is it true that if $0\le d \le n-1$, $p\ge 2n$ and we have $A \subset {\mathbb Z}_p$ of size $n$ and ${\mathcal F} \subset A \times A$ with
$d({\mathcal F}) \le d$, then the
${\mathcal F}$-restricted sum $A+_{\mathcal F} A$ has at least $2n-2d-1$ elements? The Cauchy--Davenport theorem is the case $d=0$, and the Erd\H{o}s--Heilbronn conjecture~\cite{ErdHeil}, made in 1964, and proved thirty years later
by Dias da Silva and Hamidoune~\cite{DiHa}, using tools from linear algebra and the representation
theory of the symmetric group, is the case $d=1$ with the very special family
${\mathcal F} = \{ (a,a): a \in A \}$. Subsequently Alon, Nathanson and Ruzsa~\cite{AlNaRu} gave a simpler
proof of the Erd\H{o}s--Heilbronn conjecture, which was based on the polynomial method. More recently, Vu and Wood~\cite{VuWood-09}  gave an entirely combinatorial proof under a weak additional condition.

\vspace{5pt}
Our next result shows that {\em both} extensions hold {\em simultaneously}, provided $p-2n$ is not too small and $d=d(n)$ does not grow too fast with $n$. As we shall see, the proof is a straightforward application of Theorem~\ref{stability_in_Zp_for_intro}. Note that of course this result implies Theorem~\ref{EH_for_intro} from the Introduction.

\begin{theorem}\label{EH-extension}
Let $(d (n))_1^{\infty}$ be a sequence of natural numbers with $d=d(n)=o(n)$, and let $\beta>0$. Then there are integers $c$
and $n_0$ such that the following holds.

\vspace{5pt}
Let $A\subset {\mathbb Z}_p$ with $n_0\le |A|=n \le (1 - \beta)p/2$, and let ${\mathcal F} \subset A \times A$ have degree $d$. Then there is a set $B'\in A^{(\le c)}$ such that $|A+_{\mathcal F} B'|\ge 2n-1-2d$.
\end{theorem}
\begin{proof}
Let $c\ge 2$ and $\varepsilon>0$ be the constants guaranteed by Theorem~\ref{stability_in_Zp_for_intro} for $\alpha=1$ and our $\beta >0$, and then let $n_0$ be such that if $n\ge n_0$ then $cd<  \min(\epsilon n, 2^{-1}\beta n)$. Finally, assume $n\geq n_0$.

\vspace{5pt}
If there is a set $B'\subset A$ with $|B'|=c$ such that $|A+B'|\ge 2n+r= 2n+cd$,
then this $B'$ will do for every ${\mathcal F}$, since
$|A+_{\mathcal F}B'|\ge 2n+cd-\sum_{b\in B'} d(b) \ge 2n+cd-|B'| d=2n$.
Otherwise, assume that
$\max_{B' \in B^{C}} |A+B'| \leq 2n-1+cd \leq 2n-1+\epsilon n$.

\vspace{5pt}
Applying Theorem~\ref{stability_in_Zp_for_intro} to the sets $A$ and $B=A$, we find that $A$ is contained in an arithmetic progression of size
$n+cd$. We may assume that the common difference of this arithmetic progression is 1, and we may also assume that
$A$ is contained in the interval $J=\{1, \dots , n+cd\}$ of ${\mathbb Z}_p$. Then $J+J$ has no wrap-around, so neither has $A+A$. Hence, writing $a_1$ for the first element of $A$ and $a_n$ for the last, the ${\mathcal F}$-restricted sum of $A$ and $B'=\{a_1, a_n\}$ has at least $2n-1-2d$ elements.
\end{proof}

\vspace{7pt}
The bound $2n-1-2d$ in Theorem~\ref{EH-extension} cannot be improved: this is the best bound not only for $|A+_{\mathcal F} B'|$ but also for $|A+_{\mathcal F}A|$, even if we demand that we have few forbidden pairs. Indeed, assuming, as we may, that $n>2 d$, this is shown by the set
\[
{\mathcal F}=\{(i, j): \ 1\le i, j \le d \ \ \text{or} \ \ n-d< i, j \le n\}
\]
of size only $d^2$ and the set $A=\{1, \dots , n\}$. Indeed, in this case the ${\mathcal F}$-restricted sum $A+_{{\mathcal F}} A$ of $A$ and $A$
is $\{d+2, \dots , 2n-d\}$.

\vspace{10pt}
Let us give another application of Theorem~\ref{stability_in_Zp_for_intro} to a restricted sum. This time we shall assume that the size of ${\mathcal F}$ is not too large. It is perhaps surprising that $|A+_{\mathcal F} A|$ can decrease considerably as we increase the number of forbidden pairs.
\begin{theorem}\label{restricted-by-small-graphs}
  Let $k=k(n)=o(\sqrt{n})$, $k\ge 2$, and $0<\beta < 1/2$. Let $c\ge 2$ and $0<\varepsilon < 1/2$ be the constants guaranteed by Theorem~\ref{stability_in_Zp_for_intro} for $\alpha=1$ and our $\beta$.
Then, for
 large enough $n$, whenever $A$ is a subset of ${\mathbb Z}_p$ with $|A|=n < (1 - \beta)p/2$ and
${\mathcal F} \subset A \times A$ satisfies $|{\mathcal F}|\le k(k-1)$, then
for some $B'\in A^{(\le c)}$ we have
\begin{equation}\label{size-restr-ineq}
|A+_{\mathcal F} B'|\ge 2n+1-2k.
\end{equation}
\end{theorem}
\begin{proof}
Let $n$ be large enough to guarantee $k^2< \varepsilon n$ and $2k^2<\beta p$.
If there is a set $B'\subset A$ with $|B'|=c$ such that $|A+B'|\ge 2n+k^2$,
then \eqref{size-restr-ineq} follows.

\vspace{5pt}
Otherwise, by Theorem~\ref{stability_in_Zp_for_intro}, the set $A$ is contained in an arithmetic progression of size $n+k^2$. Assume, as we may, that $A=\{a_1, \dots , a_n\}\subset \{1, \dots , n+k^2\}$, with $a_i<a_{i+1}$. Since $2(n+k^2)<p$, the sum $A+A$ has no `wraparound'. Set $d_i=d(a_i)$, $s=\min_{1\le i \le n} \big(i+d_i\big)$ and $t=\min_{1\le i \le n} \big(n+1-i+d_i\big)$. Then, crudely, $s<n/2<t$,
\[
d_1\ge s-1, \ \ \ \ \ \ d_2 \ge s-2, \ \ \ \ \dots , \ \ \ \ d_{s-1}\ge 1,
\]
and
\[
d_n\ge t-1, \ \ \ d_{n-1} \ge t-2, \ \ \ \dots , \ \ \ d_{n-t+2}\ge 1.
\]
This implies that
\[
s(s-1)/2+t(t-1)/2 \le \sum_1^n d_i= |{\mathcal F}| \le k(k-1),
\]
so $s+t\le 2k$.

\vspace{5pt}
Finally, if $j$ and $h$ are suffices where the minima are attained then, using $\{a_j, a_h\}$ for the second summand,
\begin{eqnarray*}
|A+_{\mathcal F} \{a_j, a_h\}|&\ge& |A+ \{a_j, a_h\}|-d_j-d_h\\
     &\ge&2n+1-j-d_j-(n+1-h)-d_h\\
     &\ge&2n+1-s-t\ge 2n+1-2k.
\end{eqnarray*}
This completes our proof.
\end{proof}

\vspace{5pt}
The bound in Theorem~\ref{restricted-by-small-graphs} cannot be improved. Indeed, let $A=[n]$ and
\[
{\mathcal F}=\{(i,j)\in A\times A: \ i+j\le k \ \ \ \text{or} \ \ \ i+j\ge 2n-k+2\}.
\]
Then $\big|{\mathcal F}\big|=k(k-1)$ and $A+_{\mathcal F} A=[k+1, 2n-k+1]$.

\vspace{7pt}
It is not impossible that the bounds we have imposed on $d(n)$ in Theorem~\ref{EH-extension}, and on $k(n)$ in Theorem~\ref{restricted-by-small-graphs} are unnecessary. We shall state these as conjectures in the last section.


\section{Applications -- Sums of Sets of Reals}

As we shall see now, our results in the discrete world have easy corollaries in the continuous world. Throughout this section we shall denote by $|\cdot|$ the Lebesgue measure on the Euclidean space $\mathbb{R}$ and Haar measure (normalized Lebesgue measure) on the circle $\mathbb{T}=\mathbb{R}/\mathbb{Z}$. Our starting point is the triviality that if $A$ and $B$ are non-empty compact sets of reals then $|A+B| \ge |A|+|B|$.

The deduction of these corollaries from the corresponding theorems is straightforward. First one reduces the statement from general compact sets to finite unions of closed intervals. Then one reduces the statement to discrete sets by taking the intersection with a sufficiently fine lattice. Let us show, for example, how one can deduce Corollary~\ref{cont_intro_three_translates_Z} from Theorem~\ref{intro_three_translates_Z}.

\begin{proof}[Proof of Corollary~\ref{cont_intro_three_translates_Z}]
We first reduce Corollary~\ref{cont_intro_three_translates_Z} to the case when $A$ and $B$ are finite unions of intervals and $|A|>|B|$.

Fix compact sets $A$ and $B$ in $\mathbb{R}$ and construct nested sequences of sets $A_n$ and $B_n$ in $\mathbb{R}$, all of which are finite unions of closed intervals, $|A_n| \ge |B_n|$, $\cap_n A_n = A$ and $\cap_n B_n = B$.

To prove
$$\max_{b_1,b_2,b_3\in B}|A+\{b_1,b_2,b_3\}|\geq |A|+|B|$$
given that
$$\max_{b_1^n,b_2^n,b_3^n \in B_n} |A_n+\{b_1^n,b_2^n,b_3^n\}| \geq |A_n|+|B_n| \geq |A|+|B|,$$
it suffices to show that
$$\max_{b_1,b_2,b_3 \in B} |A+\{b_1,b_2,b_3\}|\geq \lim_{n \rightarrow \infty} \max_{b_1^n,b_2^n,b_3^n \in B_n} |A_n+\{b_1^n,b_2^n,b_3^n\}|.$$

For this, note that the points $b_1^n,b_2^n,b_3^n$ are contained inside the compact set $B_1$. Therefore, we can restrict to convergent subsequences $b_1^{n_i} \rightarrow b_1^0$, $b_2^{n_i} \rightarrow b_2^0$ and $b_3^{n_i} \rightarrow b_3^0$. As the sets $B_i$ are compact, it follows that $b_1^0,b_2^0,b_3^0 \in \cap_i B_{n_i}=B$. To conclude observe that the function $f(x_1,x_2,x_3)=|A+\{x_1,x_2,x_3\}|$ is continuous and that for all $x_1,x_2,x_3$ we have $|A+\{x_1,x_2,x_3\}|\geq |A_n+\{x_1,x_2,x_3\}|-3|A_n \Delta A|$.

\vspace{5pt}
Now we further reduce Corollary~\ref{cont_intro_three_translates_Z} to Theorem~\ref{intro_three_translates_Z}.
Fix sets $A$ and $B$ in $\mathbb{R}$ that are finite unions of intervals and $|A|>|B|$, and construct sets $A_n=A\cap n^{-1}\mathbb{Z}$ and $B_n=B\cap n^{-1}\mathbb{Z}$.

To prove
\[
\max_{b_1,b_2,b_3\in B}|A+\{b_1,b_2,b_3\}|\geq |A|+|B|
\]
given that
\[
\max_{b_1^n,b_2^n,b_3^n \in B_n} |A_n+\{b_1^n,b_2^n,b_3^n\}| \geq |A_n|+|B_n|-1,
\]
it suffices to show that, for $n$ large, $|A_n+\{b_1^n,b_2^n,b_3^n\}| =n|A+\{b_1^n,b_2^n,b_3^n\}|-O(1)$, \
$|A_n|=n|A|+O(1)$ \ and \ $|B_n|=n|B|+O(1)$.
For this, recall that $A$ is a finite union of intervals and for an interval  $X$ it is trivial that $|X|=|X\cap n^{-1}\mathbb{Z}| +O(1)$. This concludes the proof.
\end{proof}

The proof of Corollary~\ref{cont_translates_of_A_Zp} from Theorems~\ref{translates_of_A_Zp} and \ref{intro_general_three_translates_Zp}
follows identical lines. Finally, we mention that there
is also a continuous version of Theorem~\ref{implication_shao1-prov}, again proved by the same method.

\begin{corollary}\label{cont_implication_shao1-prov}
For all $K$ and $\varepsilon >0$ there is an integer $c$ such that for all $d$ the following holds.
Let $A$ and $B$ be non-empty compact subsets of the torus $\mathbb{T}^d$. Then there are compact subsets $A^* \subset A$
and $B^* \subset B$, with $|A^*|\ge (1-\varepsilon)|A|$ and
$|B^*|\ge (1-\varepsilon)|B|$, such that if we select points $b_1,\hdots,b_c$ uniformly at random from
$B^*$ then
\[
\ \ \ \ \ \ \ \ \ \  \E |A^*+ \{b_1,\hdots,b_c\}|
\geq \min\big( (1-\varepsilon)|A^*+ B^*|,\ K |A^*|,\ K |B^*| \big). \ \ \ \ \ \ \ \ \ \ \square
\]
\end{corollary}

\section{Open Problems and Conjectures}

\vspace{5pt}
Perhaps the most intriguing open problem is the question of whether, in
Theorem~\ref{translates_of_A_Zp}, the number
of translates can be taken to be three, if the sizes of $A$ and $B$ are a sufficiently small
multiple of $p$. We
believe that this should be the case.

\vspace{5pt}
{\bf Conjecture 1. }{\em
There exists $0< \alpha < 1/2$ such that whenever
$A$ and $B$ are subsets of $\mathbb{Z}_p$ with $|A|=|B|\leq \alpha p$,
there exist $b_1, b_2, b_3 \in B$ such that $|A+\{b_1,b_2,b_3\}| \ge |A|+|B|-1.$
}

\vspace{5pt}
In terms of inverse results, we have not considered any questions about inverse forms of our starting theorem about
three translates in $\mathbb{Z}$, Theorem~\ref{intro_three_translates_Z}. We will address this in a forthcoming paper.


\vspace{5pt}
It would also be interesting to know if the influence of the structure of $B$ on our results is
independent of scaling. For example, we suspect that the following is the case.

\vspace{5pt}
{\bf Conjecture 2. }{\em
  Let integers $c$ and $n$  and a set $B\subset {\mathbb Z}$ be fixed. Suppose for every set
$A\subset {\mathbb Z}$ of size $n$ there are
$c$ elements $b_1, \dots , b_c \in B$ such that $|\cup_{i=1}^c (A+b_i)|\ge m$. Then the same holds if
$B$ is replaced by
$\lambda B$ for any integer $\lambda$.
}

\vspace{10pt}
Turning to restricted sums, the main question is whether the results in Section 7 hold without the perhaps
artificial restrictions on the value of $d(n)$, the maximum degree of the set of forbidden pairs.

\vspace{5pt}
{\bf Conjecture 3. }{\em
  For every $\beta>0$ there
  are constants $c$ and $n_0$ such that the following holds. If $A$ is a subset of ${\mathbb Z}_p$ with
  $n_0 \le |A|=n < (1-\beta)p/2$, and ${\mathcal F} \subset A \times A$ with $d({\mathcal F}) \le d$, then there is a set $B'\in A^{(\le c)}$ such that the ${\mathcal F}$-restricted sum $A+_{{\mathcal F}} B'$ has at least $2n-2d-1$ elements.
}

\vspace{10pt}
Theorem~\ref{EH-extension} tells us that this conjecture holds for $d=d(n)=o(n)$, and
pedestrian arguments can be used to prove it for $d=n-1, n-2$ and $n-3$. A very
interesting case would be $d=\lfloor n/2 \rfloor$, say.

\vspace{10pt}
This conjecture has a good many variants, stronger forms and weaker: here we shall state only some.

\vspace{10pt}
Firstly, does Conjecture 3 actually hold for $n_0=1$ ?
Secondly, does it hold for two sets, $A, B \in {\mathbb Z}_p$ with $|A|=|B|=n$, with
${\mathcal F} \subset A \times B$ having $d({\mathcal F}) \le d$?
It is important to note that such questions are unknown even if we drop the condition that we are passing to a subset $B'$ of $B$. Indeed, the case $d=1$ (for general sets $A$ and $B$ with $|A|=|B|=n$) is
a beautiful open problem of Lev~\cite{Lev-00b}.
Thirdly, does Conjecture 3 hold for
sets $A$ and $B$ of comparable sizes? And what happens if the sizes are incomparable?
Fourthly, what is the minimal value of $c$ for which Conjecture 3 and its relatives hold?

\vspace{10pt}
We mention that there are many questions about how the various parameters in the results relate to each
other: it would be very nice to find out the correct dependences. For example, how does $c$ depend on $\alpha$ and $\beta$ in Theorem~\ref{translates_of_A_Zp}? How does $\alpha$ depend on $\beta$ in
Theorem~\ref{intro_general_three_translates_Zp}? How small is $\varepsilon$ in
Theorem~\ref{stability_in_Zp_for_intro}? We do not see how to answer any of these questions.

\vspace{5pt}
Turning again to ${\mathbb Z}$, could it be that Theorem~\ref{intro_three_translates_Z} holds
even when $A$ is slightly smaller than $B$?

\vspace{5pt}
{\bf Question 4. }{\em
  When $A$ and $B$ are subsets of ${\mathbb Z}$ with $|B|=n$, how much smaller than $n$ can $|A|$
  be if we can still guarantee that $\max_{b_1,b_2,b_3 \in B}\bigg|A+\{b_1,b_2,b_3\}\bigg| \ge
 |A|+|B|-1$?
  }

\vspace{5pt}
Another interesting avenue of research would be to study how large a set $A + B'$ can
be in terms not of the minimum sumset bound but of the actual sumset size $|A+B|$. In other words,
how much do our bounds improve if we happen to know that $A+B$ is larger than its theoretical
minimum value? There is no chance of a result in this precise form, since $|A+B|$ may be large because $A$ has a small subset scattered in the ground set, so that $|A+B|$ is large because of a tiny subset of $A$. Note that Theorem~\ref{implication_shao1-prov} is a result in this direction. But could the
following be true? If so, what is the dependence of $c$ on $\varepsilon$ and $\lambda$?

\vspace{5pt}
{\bf Question 5. }{\em
Given constants $\varepsilon >0$ and
$\lambda >0$, is there a constant $c>0$ such that the following assertion
holds?
If $A, B \subset {\mathbb Z}_p$ with $|A|=|B|=n$ and $|A+B|\le \lambda
n$ then there are sets $A'\subset A$ and $B'\subset B$ such that $|A'|\ge
(1-\varepsilon)n$, $|B'|\le c$ and $|A'+B'|\ge (1-\varepsilon) |A'+B|$.
}

\vspace{5pt}
It might even be possible that this holds in all abelian groups. In general, as in Section 8 many of our
results in
${\mathbb Z}_p$ do of course tranfer to results in the circle ${\mathbb T}$, but for other abelian groups (continuous or
discrete) the situation is still very open.

\vspace{5pt}
Finally, there are many questions that one could ask in higher dimensions, so concerning the groups ${\mathbb Z}_p^d$
or ${\mathbb T}^d$. Perhaps the cleanest question is in the continuous torus.

\vspace{5pt}
{\bf Conjecture 6. }{\em
For each dimension $d$ there is a constant $c$ such that the following holds. Whenever
$A$ and $B$ are non-empty compact subsets of $\mathbb{T}^d$ with $|A|=|B| \leq 1/3$, there
exist $b_1, \hdots, b_c \in B$ such that $|A+\{b_1, \hdots, b_c\} |\geq |A|+|B|$.
}

\section{Appendix 1: Proofs of some results from Section 3}

We collect here the proofs of three results from Section 3.

\begin{proof}[Proof of Theorem~\ref{equality_four_translates}]
For $|A|=|B|\le 3$ this is trivial, so we may assume that $|A|=|B|\ge 4$.
Since the minimum
$
\min_{b \in B\setminus \{m\}}\bigg|A+\{0,b,m\}\bigg|=\bigg|A\cup (A+m)\bigg|
$
is attained at $b=0$, we have
$
\E_{b \in B \setminus \{0,m\}} \bigg|A +\{0,b,m\}\bigg| \geq \E_{b \in B \setminus \{m\}} \bigg|A+\{0,b,m\}\bigg|,
$
with equality if and only if
$
\bigg|A+\{0,b,m\}\bigg|=\bigg|A\cup (A+m)\bigg|
$
for every $b\in B$. By \eqref{thm1.0.0},
\begin{equation}\label{thm23.0.1}
    \E_{b \in B \setminus \{0,m\}} \bigg|A + \{0,b,m\}\bigg| \geq \E_{b \in B \setminus \{m\}} \bigg|A+ \{0,b,m\}\bigg|\ge |A|+|B|-1
\end{equation}
and, by our assumption \eqref{thm23.0.0}, both inequalities in \eqref{thm23.0.1} are equalities. Consequently,
\[
\bigg|A+\{0,b,m\}\bigg| = \bigg|A\cup (A+m)\bigg|=|A|+|B|-1
\] for every $b\in B$ and so
$A+B \subset A\cup (A+m) \subset A+B$ and
$|A+B| = |A|+|B|-1$.
Since $|A|=|B| \geq 4$,  it follows, for example by Freiman's $3k-4$ theorem (Theorem~\ref{3k-4}), that $A$ and $B$ are arithmetic progressions with the same difference.
\end{proof}

\begin{proof}[Proof of Theorem~\ref{technical_three_translates_Z}]
The proof is based on Lemma~\ref{lem8.1} and Lemma~\ref{lem8.2specialcase}, which combine to give
\begin{align*}
\E_{b\in B \setminus \{m\}}\bigg|A+\{0,b,m\}\bigg| &\geq
|A|+|\widetilde{A}| +|A| \max\bigg(0,\frac{|\widetilde{B}| - |\widetilde{A}|}{|\widetilde{B}|}\bigg)\\
&=|A|+|\pi_m(A)|+|A|\max\bigg(0, \frac{|B|-1-|\pi_m(A)|}{|B|-1} \bigg).
\end{align*}
If we assume that $|A|\geq |B|-1$ and $|\pi_m(A)| \geq |B|-1$ then we have
\begin{equation*}
    \E_{b\in B \setminus \{m\}}\bigg|A+\{0,b,m\}\bigg| \geq |A|+|\pi_m(A)|\\
    \geq |A|+|B|-1.
\end{equation*}
If we assume that $|A|\geq |B|-1$ and $|\pi_m(A)| \leq |B|-1$ then we have
\begin{align*}
    \E_{b\in B \setminus \{m\}}\bigg|A+\{0,b,m\}\bigg| &\geq |A|+|\pi_m(A)|+ |A|\frac{|B|-1-|\pi_m(A)|}{|B|-1}\\
    &\geq |A|+|\pi_m(A)|+|B|-1-|\pi_m(A)|=|A|+|B|-1.
\end{align*}
Finally, if we assume $|A|\leq |B|-1$, which implies $|\pi_m(A)| \leq |B|-1$, then we obtain
\[
    \E_{b\in B \setminus \{m\}}\bigg|A+\{0,b,m\}\bigg| \geq |A|+|\pi_m(A)|+ |A|\frac{|B|-1-|\pi_m(A)|}{|B|-1}
    = 2|A|+|\pi_m(A)|\frac{|B|-1-|A|}{|B|-1}.
\]

\vspace{5pt}
This concludes the proof of Theorem~\ref{technical_three_translates_Z}.
\end{proof}

\begin{proof}[Proof of Theorem~\ref{strengthening_of_three_translates}]

We begin with a construction and continue with some preliminary estimates. Partition $A$ into two parts,
$A=A_1\sqcup A_2 $,
such that
$\widetilde{A_2}=\widetilde{A} \text{ and }  |A_2|=|\widetilde{A}| \text{ and } |A_1|=|A|-|\widetilde{A}|.$
Lemma~\ref{lem8.1} tells us that $\pi_m\bigg(A\cup (A+m)\bigg)= \widetilde{A}$
and $\bigg|A\cup (A+m) \bigg| \geq |A|+|\widetilde{A}|$, and Lemma~\ref{lem8.2} gives
\[
\E_{b\in B \setminus \{m\}} \bigg|(A_1+b) \setminus \pi_m^{-1}(\widetilde{A})\bigg| \geq
(|A|-|\widetilde{A}|) \max\bigg(0,\frac{|\widetilde{B}|-|\widetilde{A}|}{|\widetilde{B}|}\bigg).
\]
Noting the bijection $\pi_m(B\setminus\{m\})=\widetilde{B},$
for $z \in \mathbb{Z}_d$ we have
\begin{eqnarray*}
    \Prob_{b \in B \setminus\{m\}}\bigg((A_2 +b) \cap \pi_m^{-1}(z)  \neq \emptyset \bigg)
    &=&\Prob_{\widetilde{b} \in \widetilde{B}}\bigg(z \in \widetilde{b}+\widetilde{A_2}\bigg)\\
    &=&\Prob_{\widetilde{b} \in \widetilde{B}}\bigg(z \in \widetilde{b}+\widetilde{A}\bigg)
    \geq \max\bigg(0,\frac{|\widetilde{B}| + |\widetilde{A}|-m}{|\widetilde{B}|}\bigg).
\end{eqnarray*}
It follows that
\begin{eqnarray*}
\E_{b \in B \setminus\{m\}}\bigg|(A_2 +b) \setminus \pi_m^{-1}(\widetilde{A})\bigg|
&=& \sum_{z\not\in \widetilde{A}}\Prob_{b \in B \setminus\{m\}}\bigg((A_2 +b) \cap \pi_m^{-1}(z)
\neq \emptyset \bigg)\\
    &\geq& (m-|\widetilde{A}|)\max\bigg(0,\frac{|\widetilde{B}| + |\widetilde{A}|-m}{|\widetilde{B}|}\bigg).
\end{eqnarray*}
Now, by Theorem \ref{technical_three_translates_Z},  we may assume that
\begin{equation*}
  \frac{(2|\pi_m(A)|-m)(m-(|B|-1))-1}{|B|-1}  > 0.
\end{equation*}
In particular, we have $2|\pi_m(A)|>m$, i.e.  $2|\widetilde{A}|>m$.
Since $|\widetilde{B}|+1\geq |\widetilde{A}|$, the last two
inequalities imply
\[
\frac{|\widetilde{B}| + |\widetilde{A}|-m}{|\widetilde{B}|} \geq 0.
\]
The inequalities in this proof can be put together to imply\[
\E_{b\in B \setminus \{m\}}\bigg|A+\{0,b,m\}\bigg|
      \geq |A|+|B|-1+\max\bigg(0,\frac{(2|\pi_m(A)|-m)(m-(|B|-1))-1}{|B|-1} \bigg),
\]
completing the proof of Theorem~\ref{strengthening_of_three_translates}.
\end{proof}

\section{Appendix 2: Proof of the conditions (1)-(6) in Lemma~\ref{lem25}}

Here we show that the construction given in the proof of Lemma~\ref{lem25} does indeed have the
claimed properties, as well as showing that the set $S$ is non-empty, as claimed.

\vspace{7pt}
{\bf Claim A. }{\em  Given $t\in S$, for $i=1,2,3$ the parameters $d_{i,t}$ and the consecutive intervals $I_{i}$ and $J_{i,t}$ satisfy conditions (1)-(5).
}
\begin{proof}

\vspace{7pt}
\noindent
Condition (1) is immediate from the construction.

For $i=1$ or $i=3$, we have $|J_{i,t}|\leq |J|-|J_{2,t}|+1$ and $|J_{i,t}|\leq |J|-k_2d_2$. From this we find that
$|J_{i,t}|\leq 2\min(|I|,|J|)$,
and so we can deduce the first part of condition $(2)$: $|J_{i,t}|\leq 2^4|I_i|$.

\vspace{7pt}
\noindent
For $i=2$, $|J_{2,t}|\leq |J|$,
which leads to the first part of condition $(2)$: $|J_{2,t}|\leq 2\alpha^{-1}|I_2|$.

\vspace{7pt}
\noindent
Returning to  $i=1$ or $i=3$,  we find that the function $q_i(x)=q_{i,x}: \  \mathbb{Z} \rightarrow \mathbb{Z}$ is strictly increasing. As $|q_{i,t}-q_{{i-1},t}| \geq |T_i|$, we find that $|q_{i,t}-q_{{i-1},t}| \geq 6^{-13}\min(|I|,|J|)$.

\vspace{7pt}
\noindent
Since $|J_{i,t}|=q_{i, t}-q_{i-1, t}+1 \geq 6^{-13}\min(|I|,|J|)$,
we obtain the second part of condition $(2)$:
$|J_{i,t}|\geq 6^{-13}|I_i|$.

\vspace{7pt}
\noindent
For $i=2$,  we have $|J_{2,t}|=q_{2,t}-q_{1,t}+1=k_2d_2+1 \ge  2^{-1}\min(|I|,|J|)$, which implies
the second part of condition $(2)$: \ $|J_{2,t}|\geq 2^{-1}\alpha |I_2|$.

\vspace{7pt}
\noindent
For $i=1,2,3$,  we have $|J_{i,t}|\geq 6^{-13}\min(|I|,|J|)$ and by $|I_{i}|\geq 2^{-3}\min(|I|,|J|)$.
Furthermore,
$|I_{i}\Delta A_i| \leq |A\Delta I|$ and  $|J_{i,t}\Delta B_{i,t}|\leq |J\Delta B|$,
and so
\[
\max\bigg(|I_{i}\Delta A_i| , |J_{i,t}\Delta B_{i,t}|\bigg)\leq \gamma \min(|I|,|J|).
\]
From the last four inequalities we obtain condition $(3)$:
\[
\max\bigg(|I_{i}\Delta A_i| , |J_{i,t}\Delta B_{i,t}|\bigg)\leq 6^{13}\gamma \min(|I_i|,|J_i|).
\]

\vspace{7pt}
We now turn to condition (4). For $i=1,2,3$, we have $q_{i,t}-q_{i-1,t}+1=k_id_{i,t}+1$ and
$d_{i,t} \geq 1$. Consequently,
\[
(k_i+1)d_{i,t}\geq q_{i,t}-q_{i-1,t}+1\geq k_id_{i,t},
\]
which implies
\begin{equation}\label{eq25.29957}
(k_i+1)d_{i,t}\geq|J_{i,t}|\geq k_id_{i,t}.
\end{equation}

\vspace{10pt}
\noindent
Since $k_1=2^{10}$ and $k_3=3^{10}$, for $i=1,2,3$ we have
\[
(3^{10}+1)d_{i,t}\geq |J_{i,t}|=2^{10}d_{i,t}, \ \ \text{and so} \ \ |I_{i}|\geq 2^{6}d_{i,t}.
\]
Also, if $i=2$  then $k_2d_{2,t}\geq 2^{-1}\min(|I|,|J|)$.
This tells us that, on the one hand, $|J_{2,t}|\geq 2^{-1}\min(|I|,|J|)$, and on the other,
$|I_{2}|\geq 2^{-1}\min(|I|,|J|)$.
Additionally, $d_{2,t}=\lfloor 6^{-1}\min(|I|,|J|)\rfloor. $
From the last three inequalities we conclude that
\[
\min(|I_2|,|J_{2,t}|)\geq 3d_{2,t}.
\]
From \eqref{eq25.0} we may also deduce that
$6^2d_{2,t}\geq \min(|I|,|J|)$, and so
\[
6^2d_{2,t}\geq\min(|I_2|,|J_2|).
\]
This completes the proof of Claim A.
\end{proof}

Now we turn to the proof that $S$ is non-empty, and that condition (6) holds.

\vspace{7pt}

We start by observing that the functions $g_1, g_3$ are strictly increasing in the second coordinate, and also that $g_2$ is injective. Indeed,                                   
recall that  $T_2$ is an interval of size $|T_2|\leq 6^{-10}d_2=k_1^{-1}k_3^{-1}d_2$, and $d_{2,t}=d_2 \neq 0$. It is easily checked that for $t,t'\in T_2$ we have                
$q_{1,t}=q_{1,t'} \text{ mod }d_{2,t}$ if and only if $t=t'$.                                                                                                                      
Note that $q_{1,t}=g_2(i_2,t) \text{ mod } d_{2,t}$. It follows that for $t,t'\in T_2$ we have $g_2({i_2,t})=g_2({i_2',t'}) \text{ mod }d_{2,t}$                                   
if and only if $t=t'$. Moreover, if $t=t'$ then, as $d_{2,t}=d_2\neq 0$, it follows that $g_2(i_2,t)=g_2(i_2',t)$ if and only if $i_2=i_2'$.                                      
                                          
Let us note the following three inequalities. First,
$|S^c| \leq \sum_{i\in[3]} |g_i^{(-1)} (J\setminus B)|$,
second, $|g_i^{(-1)} (J\setminus B)| \leq k_i|J\setminus B|$,
and third,  for $i=2$ we have
$|g_2^{(-1)} (J\setminus B)| \leq |J\setminus B|$.
From these three inequalities we deduce that $|S^c|\leq (k_1+k_3+1)|J\setminus B|$,
and so, by \eqref{eq25.-1}, we have
$|S^c|\leq \gamma (k_1+k_2+1) \min(|I|,|J|)$.
Finally, the choice of $k_1$ and $k_3$ implies that $S$ is non-empty.

\vspace{10pt}
{\bf Claim B. }{\em  Given $t\in S$, for each $i=1,2,3$ the parameters $d_{i,t}$ and intervals $J_{i,t}$ satisfy condition (6).
}

\begin{proof}
Our task is to check that
\[
\bigcup_{i\in[3]} [q_{i-1,t},q_{i,t}] \cap \{q_{i,t} \text{ mod } d_{i,t} \} \subset B.
\]
Note the following four identities:
\[
[q_{0,t}, q_{1,t}] \cap \{q_{1,t} \text{ mod } d_{1,t}\}= \{q_{0,t} + id_{1,t} \text{ : }i \in [0,k_1] \},
\]
\[
[q_{1,t},q_{2,t}] \cap \{q_{2,t} \text{ mod } d_{2,t}\}= \{q_{1,t} + id_{2,t} \text{ : }i \in [0,k_2] \},
\]
\[
[q_{2,t},q_{3,t}] \cap \{q_{3,t} \text{ mod } d_{3,t}\}= \{q_{3,t} - id_{3,t} \text{ : }i \in [0,k_3] \},
\]
and
\[
q_{0,t}+k_1d_{1,t}=q_{1,t}+0d_{2,t}.
\]
By \eqref{eq25.29963} and \eqref{eq25.29965} it follows that
$$ \bigcup_{i\in[3]} [q_{i-1,t},q_{i,t}] \cap \{q_{i,t} \text{ mod } d_{i,t} \} \setminus \{q_{0,t}, q_{3,t}\} \subset B.$$
Finally, by  \eqref{eq25.-1} and the definition of the parameters $q_{i,t}$, we have
\[
\{q_{0,t}, q_{3,t}\}=\{q_l,q_r\} \subset B.
\]
This establishes Claim B.
\end{proof}

\vspace{5pt}
This finishes the proof of Lemma~\ref{lem25}, since 
by Claim A, Claim B and the fact that $S$ is non-empty we conclude that there exists $t$ for which the parameters $d_{i,t}$ and the consecutive intervals $I_i$ and $J_{i,t}$
satisfy conditions $(1)$-$(6)$. Moreover, it is straightforward to check that each step of the construction is independent of the set $A$.

\end{document}